\documentclass[12 pt]{amsart}
\usepackage{amscd,amssymb,amsmath,amsthm}
\input xy
\xyoption{all}
\newtheorem{Proposition}{Proposition}

\newtheorem{Lemma}{Lemma}
\newtheorem{Theorem}{Theorem}
\newtheorem{Corollary}{Corollary}
\newtheorem{Conjecture}{Conjecture}

\newcommand{\proj}{\mathbb{P}}
\newcommand{\Z}{\mathbb{Z}}
\newcommand{\barr}{\overline}
\newcommand{\rarr}{\rightarrow}
\newcommand{\oh}{{\mathcal{O}}}
\newcommand{\com}{\mathbb{C}}
\newcommand{\Q}{\mathbb{Q}}
\newcommand{\R}{\mathbb{R}}

\newcommand{\lan}{\langle}
\newcommand{\ran}{\rangle}

\newcommand{\ZZ}{{\mathbb{Z}}}

\newcommand{\bpf}{\noindent {\em Proof.} }
\newcommand{\epf}{\qed \vspace{+10pt}}

\begin{document}

\title{Gromov-Witten theory and
Noether-Lefschetz theory}
\author{D. Maulik and R. Pandharipande}
\dedicatory{Dedicated to J. Harris on the occasion of his 60th birthday}
\date{ September 2012}
\maketitle

\begin{abstract}
Noether-Lefschetz divisors in the moduli of $K3$ surfaces are the
loci corresponding
to Picard rank at least 2. 
We relate the degrees of the Noether-Lefschetz divisors
in 1-parameter families of $K3$ surfaces  
to the Gromov-Witten theory of the 3-fold total space.
The reduced $K3$ theory and the Yau-Zaslow formula play
an important role. We use
results of Borcherds and Kudla-Millson for $O(2,19)$ lattices 
to determine 
the Noether-Lefschetz degrees in classical families of
$K3$ surfaces of degrees 2, 4, 6 and 8.
For the quartic $K3$ surfaces,  
the Noether-Lefschetz degrees 
are proven to be the Fourier coefficients of an explicitly
computed modular form of weight 21/2 and
level 8. The interplay with mirror symmetry is discussed.  We close
with a conjecture on the Picard ranks of moduli spaces of $K3$ surfaces.
\end{abstract}

\setcounter{tocdepth}{1}
\tableofcontents

\setcounter{section}{-1}
\section{Introduction}

\subsection{K3 families}
Let $C$ be a nonsingular complete curve, and let 
$$\pi: X \rarr C$$
be a 1-parameter family of nonsingular quasi-polarized $K3$ surfaces.
Let $L\in \text{Pic}(X)$ denote the quasi-polarization of
 degree
$$\int_{K3} L^2=l\in 2\mathbb{Z}^{>0}.$$
The family $\pi$ yields a morphism,
$$\iota_\pi: C \rarr \mathcal{M}_{l},$$
to the 19 dimensional 
moduli space of quasi-polarized $K3$ surfaces of
degree $l$. A review of the definitions can be found in Section \ref{nlnums}.

\subsection{Noether-Lefschetz numbers}
Noether-Lefschetz numbers are defined
by the intersection of $\iota_\pi(C)$ with Noether-Lefschetz 
divisors in $\mathcal{M}_l$. Noether-Lefschetz divisors
can be described
via Picard lattices or Picard classes.
We briefly review the two approaches.

Let $(\mathbb{L},v)$
be a rank 2 integral lattice 
with an even symmetric bilinear form
 $$\langle, \rangle : \mathbb{L} \times \mathbb{L} \rarr \mathbb{Z}$$
and a
distinguished primitive vector $v\in \mathbb{L}$ 
satisfying
$$\langle v,v \rangle =l.$$
The invariants of $(\mathbb{L},v)$ 
are the discriminant $\bigtriangleup\in \mathbb{Z}$ and the coset
$$\delta \in \left(\frac{\mathbb{Z}}{l\mathbb{Z}}\right)/\pm.$$
If the data are presented as  
$$\mathbb{L}_{h,d}=
 \left( \begin{array}{cc}
l & d  \\
d & 2h-2  \end{array} \right), \ \ v= 
\left( \begin{array}{c}
1   \\
0   \end{array} \right),
$$
then
the discriminant is
$$\bigtriangleup_l(h,d)= -\det
 \left| \begin{array}{cc}
l & d  \\
d & 2h-2  \end{array} \right| =d^2- 2lh+2l$$
and the coset is
$$\delta =  d \ \text{mod} \ l\ \  \in \left(\frac{\mathbb{Z}}{l\mathbb{Z}}\right)/\pm.$$

Two lattices $(\mathbb{L}_{h,d},v)$ and $(\mathbb{L}_{h',d'},v')$ are
equivalent if and only if 
$$\bigtriangleup_l(h,d)= \bigtriangleup_l(h',d') \ \ \text{and} \ \
\delta_{h,d}=\delta_{h',d'}.$$
However, not all pairs $(\bigtriangleup,\delta)$ are
realized.

The first type of Noether-Lefschetz divisor is defined by specifying a
Picard lattice. Let 
$$P_{\bigtriangleup,\delta} \subset \mathcal{M}_l$$
be the closure of the locus of
quasi-polarized $K3$ surfaces $(S,L)$ of degree $l$ for which
$(\text{Pic}(S),L)$ is of rank 2 with 
discriminant
$\bigtriangleup$ and coset $\delta$.
By the Hodge index theorem, $P_{\bigtriangleup,\delta}$ is empty unless
$\bigtriangleup>0$.

The second type of Noether-Lefschetz divisor is defined by
specifying a Picard class. In case $\bigtriangleup_l(h,d)> 0$,
let
$${D}_{h,d}\subset \mathcal{M}_l$$
have support on the locus of quasi-polarized $K3$ surfaces $(S,L)$ for which there
exists a class
$\beta\in \text{Pic}(S)$
satisfying
\begin{equation*}
\int_{S} \beta^2  = 2h-2 \ \ \text{and} \ \ \int_{S} \beta \cdot L
 = d.
\end{equation*}
More precisely, $D_{h,d}$ is the weighted sum
\begin{equation}\label{c67}
D_{h,d}= \sum_{\bigtriangleup,\delta} \mu(h,d\ |\bigtriangleup,\delta)\cdot 
[P_{\bigtriangleup,\delta}]
\end{equation}
where the multiplicity
$$\mu(h,d\ |\bigtriangleup,\delta)\in \{0,1,2\}$$
 is defined to be the number of elements $\beta$ of the
lattice $(\mathbb{L},v)$ associated to $(\bigtriangleup,\delta)$
satisfying
\begin{equation} \label{ggggg}
\langle \beta,\beta\rangle  = 2h-2 \ \ \text{and} \ \ \langle \beta,v\rangle 
 = d.
\end{equation}
If no lattice corresponds to $(\bigtriangleup,\delta)$, the multiplicity
$\mu(h,d\ |\bigtriangleup,\delta)$ vanishes and $P_{\bigtriangleup,\delta}$
is empty. If the multiplicity is nonzero, then
$$\bigtriangleup | \bigtriangleup_l(h,d).$$
Hence, the sum  on the right of
\eqref{c67} has only finitely many terms.

As relation \eqref{c67} is easily seen to be triangular,
the divisors $P_{\bigtriangleup,\delta}$ and $D_{h,d}$ are essentially
equivalent. However, the divisors $D_{h,d}$ will be seen to
have better formal properties.

A natural approach to studying the divisors $D_{h,d}$ is via
intersections with test curves. In case $\bigtriangleup_l(h,d)>0$,
the Noether-Lefschetz number $NL^\pi_{h,d}$ 
is
the  classical intersection
product
\begin{equation}\label{def11}
NL^\pi_{h,d} =\int_C \iota_\pi^*[D_{h,d}].
\end{equation}
If $\bigtriangleup_l(h,d)<0$, the divisor $D_{h,d}$ vanishes
by the Hodge index theorem.
A definition of $NL^\pi_{h,d}$ for all values $\bigtriangleup_l(h,d)\geq 0$ 
is 
given by classical intersection theory in the period domain for
$K3$ surfaces     
in Section \ref{nlnums}.

The divisibility of a nonzero element $\beta$ of
a lattice is the maximal positive integer $m$ dividing $\beta$.
Refined divisors $D_{m,h,d}$ are defined by 
\begin{equation*}
D_{m,h,d}= \sum_{\bigtriangleup,\delta} \mu(m,h,d\ |\bigtriangleup,\delta)\cdot 
[P_{\bigtriangleup,\delta}]
\end{equation*}
where the multiplicity
$$\mu(m,h,d\ |\bigtriangleup,\delta)\in \{0,1,2\}$$
 is the number of elements $\beta$ of divisibility $m$ of the
lattice $(\mathbb{L},v)$ associated to $(\bigtriangleup,\delta)$
satisfying \eqref{ggggg}. Refined
Noether-Lefschetz number are defined by
\begin{equation*}
NL^\pi_{m,h,d} =\int_C \iota_\pi^*[D_{m,h,d}].
\end{equation*}

\subsection{Invariants}\label{nlll}
We will study three types of invariants associated to a  
1-parameter family $\pi$ of quasi-polarized $K3$ surfaces
in case the total space $X$ is nonsingular:
\begin{enumerate}
\item[(i)] the Noether-Lefschetz numbers  of $\pi$,
\item[(ii)] the Gromov-Witten invariants of $X$,
\item[(iii)] the reduced Gromov-Witten invariants of the
$K3$ fibers. 
\end{enumerate}
The Noether-Lefschetz numbers (i) are classical intersection
products while the Gromov-Witten invariants (ii)-(iii) 
are quantum in origin.

The Gromov-Witten invariants (ii) of the  3-fold $X$ and the
reduced Gromov-Witten invariants (iii) of a $K3$ surface
are defined via 
integration against virtual classes of moduli spaces
of stable maps. We view both of these Gromov-Witten theories in terms
of the associated BPS state
counts defined by Gopakumar and Vafa \cite{GV1,GV2}.

Let $n_{g,d}^X$ denote the Gopakumar-Vafa invariant of $X$
of genus $g$ for $\pi$-vertical curve classes of degree $d$
with respect to $L$. Let
$r_{g,m,h}$ denote the Gopakumar-Vafa reduced $K3$ invariant of 
genus $g$ and curve class $\beta\in H_2(K3,\mathbb{Z})$
of divisibility $m$ 
satisfying 
$$\int_{K3} \beta^2 =2h-2.$$
A review of these quantum invariants is presented in
Section \ref{gwrev}.

A geometric result intertwining the invariants (i)-(iii)
is derived in Section \ref{tmm} 
by a comparison of
the reduced and usual deformation theories of maps of curves
to
the $K3$ fibers of $\pi$.

\begin{Theorem} \label{oooo}
 {\em For $d>0$,}
$$n_{g,d}^X= \sum_{h} \sum_{m=1}^\infty
r_{g,m,h}\cdot  NL_{m,h,d}^\pi.$$
\end{Theorem}

Theorem \ref{oooo} is the main geometric result of
the paper. The proof is given in Section \ref{tmm}.

\subsection{Applications}
Since Theorem 1 relates three distinct geometric invariants,
the result can be effectively used in several directions.

An application for studying reduced invariants of $K3$ surfaces is given in \cite{KMPS}.
A central conjecture discussed in Section \ref{bpsk3}
is the {\em independence}{\footnote{If $m^2$ does not divide $2h-2$,
then $r_{g,m,h}=0$. The independence is conjectured
only when $m^2$ divides $2h-2$. When we write $r_{g,m,h}$,
the divisibility condition is understood to hold.}}
 of $r_{g,m,h}$ on $m$. 
In genus 0, the independence is the non-primitive Yau-Zaslow
conjecture proven in \cite{KMPS} as a consequence of Theorem 1.

The approach taken there is the following.  For a specific
1-parameter family of $K3$ surfaces, known in the physics literature as the STU model,
the BPS states $n^{STU}_{0,d}$ are known by proven
mirror transformations and the Noether-Lefschetz
numbers $NL_{m,h,d}^{STU}$ can by exactly determined. Theorem 1 is then
used in \cite{KMPS} to solve for $r_{0,m,h}$:
$$r_{0,m,h} = r_{0,1,h}, \ \ \ \  \ \sum_{h\geq 0} r_{0,1,h} = 
\prod_{n\geq 1} \frac{1}{(1-q^n)^{24}} \ .$$
The genus 1 results
$$r_{1,m,h} = r_{1,1,h} = -\frac{h}{12}\  r_{0,1,h}$$
are an easy consequence, see Section \ref{bpsk3}.
We write $r_{g,m,h} = r_{g,h}$ independent of $m$ for $g=0,1$.

Using \cite{KMPS}, the genus 0 and 1 specialization takes
a much simpler form.
\begin{Corollary} \label{ooooo}
 {\em For $g\leq 1$ and $d>0$,}
$$n_{g,d}^X= \sum_{h=g}^\infty 
r_{g,h}\cdot  NL_{h,d}^\pi.$$
\end{Corollary}

By Corollary \ref{ooooo}, the Gromov-Witten invariants 
$n_{g,d}^X$  are completely determined by the
Noether-Lefschetz numbers of $\pi$ for any 1-parameter
family of quasi-polarized $K3$ surfaces.
The result may be viewed as giving a fully classical
interpretation of the Gromov-Witten invariants
of $X$ in $\pi$-vertical classes.

Theorem 1 
can also be used to constrain the
 Noether-Lefschetz degrees themselves.
An important approach to the Noether-Lefschetz numbers (already used
in the STU calculation) is via results of 
Borcherds \cite{borch} and Kudla-Millson  \cite{kudmil}.
The Noether-Lefschetz numbers of $\pi$ are proven to be  the
Fourier coefficients of a vector-valued modular 
form.{\footnote{While the paper \cite{borch,kudmil}
have considerable overlap, we will follow the
point of view of Borcherds.}}
For several classical families of $K3$ surfaces,
Corollary \ref{ooooo} in genus 0 provides
an {alternative} method of calculating the Noether-Lefschetz
numbers via the invariants $n_{0,d}^X$.
Together, we obtain a remarkable sequence of identities intertwining
hypergeometric series from mirror transformations (calculating
$n_{0,d}^X$) 
and modular forms. The Harvey-Moore identity \cite{hm1} for the
STU model is a special case.

As a basic example, we provide a complete calculation of
the Noether-Lefschetz numbers for
the family of
$K3$ surfaces  determined by a Lefschetz pencil of
quartics in $\proj^3$. 
The required mirror symmetry calculations (iii) for the
quartic pencil have long been established rigorously \cite{giv1,giv2}.
We give the derivation of the Noether-Lefschetz numbers via 
Gromov-Witten calculations
in Section \ref{quarticcalc}.
The resulting hypergeometric-modular identity follows 
immediately in Section \ref{modid}.
A second approach to calculating Noether-Lefschetz numbers directly
via 
more sophisticated
modular form techniques is explained for quartics and 
several other classical families in Section \ref{dnl}.

Once the Noether-Lefschetz numbers are calculated
for the 1-paramet\-er family $\pi$, Corollary \ref{ooooo}
yields the genus 1 Gromov-Witten invariants of $X$
in $\pi$-vertical classes. There are very few
methods for the exact calculation of genus 1 invariants in
Calabi-Yau geometries.{\footnote{See \cite{Zinger} for
a different mathematical approach to genus 1 invariants
for complete intersections.}} Corollary \ref{ooooo} provides
a new class of complete solutions.

\subsection{Heterotic duality}
In rather different terms,
approach (i)-(iii) was pursued in the string theoretic
work of Klemm, Kreuzer, Riegler, and Scheidegger \cite{germans} 
with the goal of calculating
the BPS counts $n_{g,d}^X$ from the genus 0 values $n^X_{0,d}$. 
Heterotic duality was used in \cite{germans} for (i) 
since the connection to the intersection theory of the Noether-Lefschetz
divisors 
$$D_{h,d}\subset \mathcal{M}_l$$
and the work of Borcherds was not 
made.
The perspective of \cite{germans} 
can be turned upside down by using Gromov-Witten
theory to calculate the Noether-Lefschetz numbers.
On the other hand, modularity allows the calculations
of \cite{germans} to be pursued in much greater generality.

In fact, the back and forth here between heterotic duality and mathematical
results is  older. Borcherds' paper on automorphic functions
\cite{borch2} which underlies \cite{borch} was motivated in part
by the work of Harvey and Moore \cite{hm1,hm2} on heterotic duality. 
The first higher genus results for $K3$ fibrations were by
Mari\~no and Moore \cite{MM}.

Finally, we mention the circle of ideas here can be considered for interesting
isotrivial families of $K3$ surfaces with double Enriques 
fibers \cite{klemmm,mp}. While heterotic duality arguments
apply there, Borcherds' result does not directly apply.

\subsection{Modular forms} \label{mform}
Let $A$ and $B$ be modular forms of weight $1/2$ and level 8,
$$A= \sum_{n\in \mathbb{Z}} q^{\frac{n^2}{8}}, \ \ \
B= \sum_{n\in \mathbb{Z}} (-1)^n  q^{\frac{n^2}{8}}.$$
Let $\Theta$ be the modular form of weight $21/2$ and level 8 
defined by
\begin{eqnarray*}
2^{22} \Theta &=&\ \ 3A^{21}-81 A^{19}B^2 -627 A^{18}B^3 -14436 A^{17}B^4 \\
& & -20007 A^{16}B^5  -169092 A^{15}B^6 -120636 A^{14}B^7\\
& &   -621558 A^{13}B^8
-292796 A^{12}B^9 -1038366 A^{11}B^{10}\\
& &  -346122 A^{10}B^{11}
-878388 A^{9} B^{12} -207186 A^8 B^{13}\\
& &  -361908 A^7 B^{14} -56364 A^6 B^{15} -60021 A^5 B^{16}\\
& &  -4812 A^4 B^{17}
-1881 A^3 B^{18} -27 A^2 B^{19}+ B^{21}.
\end{eqnarray*}
We can expand $\Theta$ as a series in $q^{\frac{1}{8}}$,
$$ \Theta = -1 +108 q +320 q^{\frac{9}{8}}+50016 q^{\frac{3}{2}}+ 76950q^2 \ldots .$$
The modular form $\Theta$  first appeared in calculations of \cite{germans}.

Let $\pi$ be the family of quasi-polarized $K3$ surfaces determined
by a Lefschetz pencil of quartics in $\proj^4$.
Let $\Theta[m]$ denote the coefficient of $q^m$ in $\Theta$.

\setcounter{Theorem}{1}
\begin{Theorem}
The Noether-Lefschetz numbers of the quartic pencil
$\pi$ are coefficients of $\Theta$,
$$NL^{\pi}_{h,d} = \Theta\left[ \frac{\bigtriangleup_4(h,d)}{8}\right].$$
\end{Theorem}

\subsection{Classical quartic geometry}

Let $V$ be a 4-dimensional $\com$-vector space. A quartic hypersurface
in $\proj(V)$ is determined by an element of $\proj(\text{Sym}^4 V^*)$. Let
$$\mathcal{U} \subset \proj(\text{Sym}^4 V^*)$$
be the Zariski open set of nonsingular quartic hypersurfaces.
Since $[S]\in \mathcal{U}$ corresponds to a polarized $K3$ surface of
degree 4, we obtain a canonical morphism
$$\phi: \mathcal{U} \rarr \mathcal{M}_4.$$ 
If $\bigtriangleup_4(h,d)> 0$, the pull-back
$$\mathcal{D}_{h,d}=\phi^{-1}(D_{h,d})\subset \mathcal{U}$$
is a closed subvariety of pure codimension 1.
As a Corollary of Theorem 2, we obtain 
a complete calculation of the degrees of
the hypersurfaces 
$$\overline{\mathcal{D}}_{h,d} \subset \proj(\text{Sym}^4 V^*).$$


\setcounter{Corollary}{1}
\begin{Corollary} If $\bigtriangleup_4(h,d)> 0$,
the degree of $\overline{\mathcal{D}}_{h,d}$ is
\begin{equation*}
\text{\em deg}(\overline{\mathcal{D}}_{h,d}) = 
\Theta\left[\frac{\bigtriangleup_4(h,d)}{8}
\right] -
\Psi\left[\frac{\bigtriangleup_4(h,d)}{8}
\right]
\end{equation*}
where the correction term is
$$\Psi = 108 \sum_{n> 0} q^{n^2}.$$
\end{Corollary}

The correction term, obtained from the contribution of the
nodal quartics, is explained in Section \ref{c22}. 
Formulas for the degrees of $$\overline{\phi^{-1}(P_{\bigtriangleup,\delta})} \subset
\proj(\text{Sym}^4 V^*)$$
are easily obtained from \eqref{c67} and a parallel 
nodal analysis.
While Corollary 2 answers
a classical question about the Hodge theory of quartic $K3$ surfaces,
the method of proof is modern.

\subsection{Outline}

In Section \ref{nlnums}, we give a precise definition of Noether-Lefschetz numbers and establish several
 elementary properties.  
The definitions of BPS invariants for 3-folds and 
reduced Gromov-Witten invariants of $K3$ surfaces 
are recalled in Section \ref{gwrev}.
Two central conjectures about the reduced theory of $K3$ surfaces are stated in Section \ref{bpsk3}.
The proof of Theorem 1 is presented in Section \ref{tmm}.

We review of the work of Borcherds on Heegner divisors and explain the application to families of $K3$ surfaces
in Section \ref{modularforms}.
The results are applied with Theorem 1 to prove Theorem 2 via mirror symmetry calculations in Section \ref{quarticcalc}.  
A direct
 approach to Noether-Lefschetz degrees for classical familes of $K3$ surfaces of degrees $2, 4, 6$, and $8$
is given  in Section \ref{dnl} via a deeper study
 of vector-valued modular forms.
Finally, in Section \ref{picr}, we state a conjecture regarding Picard ranks
of moduli spaces of $K3$ surfaces of degree $l$ .

 \subsection{Acknowledgments}
Discussions with A. Klemm about the calculations in \cite{germans} played
a crucial role. We are grateful to D. Huybrechts for
a careful reading of the paper. 

We thank R. Borcherds, J. Bruinier, J. Bryan, B. Conrad,
 I. Dolgachev, S. Grushevsky,  E. Looijenga, G. Moore, 
K. Ranestad,
P. Sarnak, E. Scheidegger,
C. Skinner,
 A. Snowden, W. Stein, G. Tian, I. Vainsencher,
and W. Zhang for conversations
about Noether-Lefschetz divisors, reduced invariants of
$K3$ surfaces, and modular forms.

D. M. was partially supported by an NSF graduate fellowship.
R.P. was partially support by NSF grant DMS-0500187 and a Packard
foundation fellowship. The paper was written in the
spring of 2007 and revised in 2009.

\section{Noether-Lefschetz numbers}\label{nlnums}
\subsection{Picard lattice}
Let $S$ be a $K3$ surface. The second cohomology of $S$ is a rank 22 lattice
with intersection form 
\begin{equation}\label{ccet}
H^2(S,\mathbb{Z}) \stackrel{\sim}{=} U\oplus U \oplus U \oplus E_8(-1) \oplus E_8(-1)
\end{equation}
where
$$U
= \left( \begin{array}{cc}
0 & 1 \\
1 & 0 \end{array} \right)$$
and 
$$E_8(-1)=  \left( \begin{array}{cccccccc}
 -2&    0 &  1 &   0 &   0 &   0 &   0 & 0\\
    0 &   -2 &   0 &  1 &   0 &   0 &   0 & 0\\
     1 &   0 &   -2 &  1 &   0 &   0 & 0 &  0\\
      0  & 1 &  1 &   -2 &  1 &   0 & 0 & 0\\
      0 &   0 &   0 &  1 &   -2 &  1 & 0&  0\\
      0 &   0&    0 &   0 &  1 &  -2 &  1 & 0\\ 
      0 &   0&    0 &   0 &   0 &  1 &  -2 & 1\\
      0 & 0  & 0 &  0 & 0 & 0 & 1& -2\end{array}\right)$$
is the (negative) Cartan matrix. The intersection form \eqref{ccet}
is even.

The {\em divisibility} of $\beta\in H^2(S,\mathbb{Z})$ is
the maximal positive integer dividing $\beta$.
If the divisibility is 1,
$\beta$ is {\em primitive}.
Elements with
equal divisibility and norm are equivalent up to orthogonal transformation 
 of $H^2(S,\mathbb{Z})$, see \cite{CTC}.

The Hodge decomposition of the second cohomology of $S$ has dimensions $(1,20,1)$,
$$H^2(S,\mathbb{Z}) \otimes_{\mathbb{Z}} \com = H^{2,0}(S,\com) \oplus
H^{1,1}(S,\com) \oplus H^{0,2}(S,\com).$$
The {\em Picard lattice} of $S$ is
$$\text{Pic}(S)= H^2(S,\mathbb{Z}) \cap H^{1,1}(S,\com).$$

\subsection{Quasi-polarization}\label{quasp}
A {\em quasi-polarization} on $S$ is a line bundle $L$ with primitive Chern
class $c_1(L)\in H^2(S,\mathbb{Z})$
satisfying
$$\int_S L^2>0 \ \ \text{and} \ \ \int_S L\cdot [C] \geq 0$$
for every curve $C\subset S$.
A sufficiently high tensor power $L^n$ of a quasi-polarization is 
base point free and determines a birational morphism 
$$S\rarr \widetilde{S}$$
contracting A-D-E configurations of
$(-2)$-curves on $S$ \cite{Sand}.
Hence, every quasi-polarized $K3$ surface $(S,L)$ is algebraic.

Let $X$ be a 
compact 3-dimensional complex manifold  equipped with a holomorphic
line bundle
$L$
and a holomorphic
map
$$\pi:X \rarr C$$
to a nonsingular complete curve. The triple $(X,L,\pi)$
is a {\em family of quasi-polarized $K3$ surfaces of degree $l$} if
the fibers $(X_\xi, L_\xi)$
are quasi-polarized $K3$ surfaces satisfying
$$\int_{X_\xi} L_\xi^2=l$$
for every $\xi\in C$.
The family $(X,L,\pi)$ yields a morphism,
$$\iota_\pi: C \rarr \mathcal{M}_{l},$$
to the moduli space of quasi-polarized $K3$ surfaces of
degree $l$.

We will often refer to the triple $(X,L,\pi)$ just
by $\pi$.
Associated to $\pi$ is the projective variety $\widetilde{X}$
obtained from the relative
quasi-polarization,
$$X \rarr \widetilde{X}\subset \proj (R^0\pi_*(L^n)^*)\rarr C,$$
for sufficiently large $n$.
The complex manifold $X$ may be a non-projective small resolution of the
singular projective variety $\widetilde{X}$.

\subsection{Period domain}
Let $V$ be a rank 22 integer lattice with intersection form $\langle,\rangle$
obtained from the second homology of a $K3$ surface,
$$V\stackrel{\sim}{=}U\oplus U \oplus U \oplus E_8(-1) \oplus E_8(-1).$$
A
1-dimensional subspace $\com\cdot \omega\in V\otimes _{\Z} \com$ satisfying
\begin{equation}
\label{w023}
\langle \omega, \omega \rangle =0 \ \ \text{and} \ \ \lan \omega,\barr{\omega} \ran > 0
\end{equation}
determines a Hodge structure of type $(1,20,1)$
on $V$, 
$$V\otimes_\Z \com  = 
V^{2,0} \oplus V^{1,1} \oplus V^{0,2} =
\com\cdot \omega \ \oplus \ (\com\cdot
 \omega \oplus \com \cdot 
\barr{\omega})^\perp \ \oplus\ \com\cdot \barr{\omega}.$$
Conversely, a Hodge structure of type $(1,20,1)$ determines a 1-dimensional
subspace $\com\cdot \omega$ satisfying \eqref{w023}.

The moduli space $M^V$ of Hodge structures of type $(1,20,1)$ on $V$ is
therefore an analytic open set
 of the 20-dimensional  nonsingular isotropic 
quadric $Q$,
$$M^V\subset Q\subset \proj(V \otimes_\Z \com).$$
The moduli space $M^V$ is the {\em period domain}.

For nonzero $\beta\in V$, let $D^V_\beta \subset M^V$ denote the
locus of Hodge structures for which $\beta\in V^{1,1}$.
Certainly,
$$D^V_\beta = M^V \cap \beta^\perp\subset \proj(V\otimes_\Z\com)$$
where $\beta^\perp$ is the linear space orthogonal to $\beta$.
Hence, $D^V_\beta$ is simply a 19-dimensional
hyperplane section of $M^V$.

\subsection{Local systems}
\label{lsy}
Let $(X,L,\pi)$ be a quasi-polarized family of $K3$ surfaces 
over a nonsingular curve $C$.
Let
$$\mathcal{V}=R^2\pi_*(\mathbb{Z})\rarr C$$
denote the rank 22 local system 
determined by the middle cohomology of the fibration 
$$\pi:X \rarr C.$$
The local system $\mathcal{V}$ is  
equipped with the fiberwise intersection form $\langle,\rangle$.

Let $\mathcal{M}^{\mathcal{V}}$ be the $\pi$-relative moduli space of Hodge structures
$$\mu:\mathcal{M}^{\mathcal{V}} \rarr C$$
with fiber
$$\mu^{-1}(\xi)= M^{\mathcal{V}_\xi}.$$
The moduli space $\mathcal{M}^{\mathcal{V}}$ is a complex manifold, and
$\mu$ is a locally trivial fibration in the analytic topology.

Duality and  homological push-forward yield a canonical  
map
$$\epsilon: \mathcal{V} \rarr H_2(X,\Z)$$
where the right side can be viewed as a trivial local system.
Let $H_2(X,\mathbb{Z})^\pi$ denote the kernel of the
projection map
$$\pi_*: H_2(X,\mathbb{Z}) \rarr H_2(C,\mathbb{Z}).$$
For $h\in \mathbb{Z}$ and 
$\gamma\in H_2(X,\mathbb{Z})^\pi$, we will define
a Noether-Lefschetz number $NL_{h,\gamma}^\pi$ for the $K3$
fibration $\pi$.

Informally,
$NL_{h,\gamma}^\pi$ 
counts the number of points $\xi\in C$ for which 
there exists an integral class 
$\beta \in V_\xi$ of type $(1,1)$
satisfying
$$\lan \beta, \beta \ran=2h-2 \ \ \text{and} \ \ 
\epsilon(\beta)=\gamma.$$
The formal definition is given in Section \ref{nnn}.

\subsection{Classical intersection}\label{nnn}
Define the relative divisor $${\mathcal{D}}
^{\mathcal{V}}_{h,\gamma} \subset \mathcal{M}^{\mathcal{V}}$$
by 
the set of Hodge structures which contain a
  class 
$\beta \in {\mathcal{V}}_\xi$ of type $(1,1)$
satisfying 
$$\lan \beta, \beta \ran=2h-2 \ \ \text{and} \ \
\epsilon(\beta)=\gamma.$$
When $\mathcal{M}^{\mathcal{V}}$ is trivialized{\footnote{We take
trivializations obtained 
from trivializing $R^2\pi_*(\mathbb{Z})$ compatibly
with $\epsilon$ .}}
over a Euclidean
open set $U\subset C$,
$$\mathcal{M}^{\mathcal {V}_U}= M^V \times U,$$
the subset
$\mathcal{D}^{\mathcal{V}}_{h,\gamma}$ restricts to
$$\mathcal{D}^{\mathcal{V}_U}_{h,\gamma} = \cup_{\beta}\  D^V_\beta \times U$$
where the union is over all
$\beta \in V$ 
 satisfying
$$\lan \beta, \beta \ran=2h-2 \ \ \text{and} \ \
\epsilon(\beta)=\gamma.$$
Hence, $\mathcal{D}^{\mathcal{V}}_{h,\gamma}\subset 
\mathcal{M}^{\mathcal{V}}$ {is} a countable union of divisors.

The Noether-Lefschetz number is defined by a tautological
intersection product.
The family $\pi$ determines a canonical section
$$\sigma: C \rarr \mathcal{M}^{\mathcal{V}}.$$
where  
$$ \sigma(\xi)=[H^{2,0}(X_\xi,\com)] \in \mathcal{M}^{\mathcal{V}_\xi}$$
is the Hodge structure determined by the
$K3$ surface $X_\xi$. 
Let
\begin{equation}\label{def22}
NL_{h,\gamma}^\pi = \int_C \sigma^*[\mathcal{D}^{\mathcal{V}}_{h,\gamma}].
\end{equation}
The divisor $\mathcal{D}^{\mathcal{V}}_{h,\gamma}$ may have 
infinitely many components. However,
by the finiteness result of Proposition 1, $NL_{h,\gamma}^{\pi}$ is
well-defined.

While $NL_{h,\gamma}^\pi$ is a classical intersection
number, an excess calculation is required in case
$\sigma(C) \subset \mathcal{D}^{\mathcal{V}}_{h,\gamma}$. The
informal counting interpretation is not always well-defined.

\begin{Proposition} 
$NL_{h,\gamma}^\pi$
is finite.
\end{Proposition}

\bpf
Let $L$ be the quasi-polarization on $X$. If there
exists a point $\xi\in C$ for which $L_\xi$ is ample, then
$L$ is $\pi$-relatively ample over an open set
of $C$.
If $L_\xi$ is never ample, then the morphism 
$$X \rarr \widetilde{X} \subset \proj(R^0\pi_*(L^n))$$
for sufficiently large $n$
contracts divisors on $X$ which intersect the generic fiber
$X_\xi$ in (-2)-curves.
After modification{\footnote{A base change of $\pi: X \rarr
C$ is
not required since the modification can be
averaged over the symmetries of the (-2)-curve configuration.}} 
of $L$ by these contracted divisors, a
new quasi-polarization $L'$ of $X$ may be obtained which is
$\pi$-relatively ample over a nonempty open set of $C$.

We assume now (after possible modification) the quasi-polarization
$L$ is $\pi$-relatively ample over a nonempty open set $U\subset C$.
Let
$$d= \int_\gamma  L $$
be the degree of $\gamma$.
Let
$$l= \int_{X_\xi} L^2_\xi>0$$
be the degree of the $K3$ fibers of $\pi$.

Let
$\beta\in \mathcal{V}_\xi$ of type $(1,1)$ satisfy
$$\lan \beta, \beta \ran=2h-2 \ \ \text{and} \ \
\epsilon(\beta)=\gamma.$$
We will prove 
$$\sigma(C)\subset \mathcal{M}^{\mathcal{V}}$$
 intersects only finitely
many components of $\mathcal{D}^{\mathcal{V}}_{h,\gamma}$.

Let $k$ be an integer satisfying
$$d+lk>0 \ \ \text{and} \ \ lk^2+2dk+2h-2>-4.$$
The first step is to show
$$\tilde{\beta} = \beta+ kc_1({L}_\xi)$$
is an effective curve class on $X_\xi$ 
by Riemann-Roch.

Let $L_{\tilde{\beta}}$ denote the unique line
bundle on $X_\xi$ with 
$$c_1(L_{\tilde{\beta}})={\tilde{\beta}}.$$
By Serre duality,
$$H^2(X_\xi,L_{\tilde{\beta}})= H^0(X_\xi, L^*_{\tilde{\beta}})^*$$
Since
$$\lan c_1(L^*_{\tilde{\beta}}), L_\xi\ran \leq -d-lk <0,$$
$h^0(X_\xi, L^*_{\tilde{\beta}})$ vanishes.
Then, by Riemann-Roch,
\begin{eqnarray*}
h^0(X_\xi, L_{\tilde{\beta}}) & \geq & 
\chi(X_\xi, L_{\tilde{\beta}})-h^2(X_\xi,L_{\tilde{\beta}})\\
&  = & \chi(X_\xi,L_{\tilde{\beta}}) \\
& = & \frac{1}{2}\lan \tilde{\beta},\tilde{\beta} \ran +2\\
& > & 0.
\end{eqnarray*}
Hence, $\tilde{\beta}$ is an effective curve class on $X_\xi$.


Consider 
first the open set $U\subset C$
over which $L$ is $\pi$-relatively ample.
Let
$$\mathcal{H} \rarr U$$
be the $\pi$-relative Hilbert scheme parameterizing
of curves in $X_{\xi\in U}$ of degree
\begin{equation*}\label{ffr}
\lan \tilde{\beta}, c_1({L}_\xi)\ran= d+lk
\end{equation*}
and Euler characteristic 
\begin{equation*}\label{rrff}
\chi(X_\xi,\oh_{X_\xi}) - \chi(X_\xi, L_{\tilde{\beta}}^*) =
-\frac{1}{2}\lan \tilde{\beta}, \tilde{\beta} \ran =
-\frac{1}{2}(lk^2+2dk+2h-2).
\end{equation*}
The scheme ${\mathcal{H}}$
is projective over $U$ and of finite type.
 
An irreducible component $\mathcal{H}_{irr}\subset 
\mathcal{H}$ either dominates $U$ or
maps to a point ${\xi}\in U$.
In the former case, 
the classes of curves represented
by $\mathcal{H}_{irr}$ yield a {\em finite} monodromy invariant
subset of $\mathcal{V}$.
In the latter case, the curves represented
by $\mathcal{H}_{irr}$ yield a single element of 
$\mathcal{V}_{{\xi}}$.

After shifting the finiteness statements back by
$kc_1({L}_\xi)$, we obtain the finiteness of the
intersection geometry
\begin{equation}\label{fd122}
\sigma(C) \cap \mathcal{D}^{\mathcal{V}}_{h,\gamma}
\end{equation}
over $U\subset C$. 
Indeed,
the dominant components $\mathcal{H}_{irr}$ correspond to
finitely many excess intersections and the non-dominant components
correspond to finitely many true intersections.

Finally consider the complement $U^c \subset C$. The complement
is a finite set. For each $\xi^c \in U^c$, let $L^c_{\xi^c}$
be an ample line bundle. The above arugment using the
ample bundles $L^c_{\xi^c}$ for the fibers $X_{\xi^c}$
shows there are finitely many intersections in \eqref{fd122}
over $U^c \subset C$ as well.

We conclude the intersection geometry is finite over all of $C$ and
the product 
$$NL_{h,\gamma}^\pi= 
\int_C \sigma^*[\mathcal{D}^{\mathcal{V}}_{h,\gamma}]$$
is well-defined.
\epf

Let $\gamma_{{L}}$ denote the push-forward of the ample
class on the fibers,
$$\gamma_{{L}}= c_1({L}) \cap [X_\xi] 
\in H_2(X,\mathbb{Z})^\pi.$$
By an elementary comparison,
$$\sigma^*[\mathcal{D}_{h,\gamma}^{\mathcal{V}}] = 
\sigma^*[\mathcal{D}_{h+d+\frac{l}{2},
\gamma+\gamma_{{L}}}^{\mathcal{V}} ].$$
We obtain the following result.

\begin{Proposition}
$NL_{h,\gamma}^\pi = NL_{h+d+{\frac{l}{2}}, \gamma+\gamma_{{L}}}^\pi.$
\end{Proposition}

The proof
of Proposition 1 show the vanishing of the
Noether-Lefschetz number for high $h$.

\begin{Proposition}
For fixed $\gamma$, the numbers $NL_{h,\gamma}^\pi$
vanish for sufficiently high $h$.
\end{Proposition}

The 
Noether-Lefschetz numbers $NL_{h,\gamma}(\pi)$ have
  non-trivial dependence on $\gamma$
despite the linear equivalence
$$D^V_\beta \cong D^V_{\beta'}$$
on $M^V$. The Noether-Lefschetz numbers involve also
the twisting of the local system $\mathcal{V}$ over $C$.

\subsection{Refinements}
The Noether-Lefschetz numbers $NL^\pi_{h,d}$ defined
in Section \ref{nlll} are obtained from the relation
\begin{equation}\label{yyy12}
NL^\pi_{h,d} = \sum_{\int_\gamma L=d} NL^\pi_{h,\gamma}.
\end{equation}
The finiteness of the sum on the right is a consequence of the
negative definiteness of the intersection matrix of divisors
in $X_\xi$ contracted by $L_\xi$.
The invariants $NL^\pi_{h,\gamma}$ may be viewed as
a refinement of $NL^\pi_{h,d}$ with the nonvanishing 
discriminant hypothesis lifted.

Further refined Noether-Lefschetz numbers may be defined
with respect to any additional monodromy invariant data. 
For example, the divisibility $m$ of an element 
$\beta\in \mathcal{V}_\xi$ is a monodromy invariant.
Let $${\mathcal{D}}
^{\mathcal{V}}_{m,h,\gamma} \subset \mathcal{M}^{\mathcal{V}}$$
be the divisor
of Hodge structures which contain a class
$\beta \in {\mathcal{V}}_\xi$ of type $(1,1)$
of divisibility $m$
satisfying 
$$\lan \beta, \beta \ran=2h-2 \ \ \text{and} \ \
\epsilon(\beta)=\gamma.$$
We define
$$NL^\pi_{m,h,\gamma} = \int_C \sigma^*[\mathcal{D}_{m,h,\gamma}].$$
The relation
\begin{equation}\label{zz23}
NL^\pi_{h,\gamma}= \sum_{m\geq 1} NL^\pi_{m,h,\gamma}
\end{equation}
certainly holds.

\subsection{Intersection theory of $\mathcal{M}_l$}
Let $v\in V$ be a vector of norm $l$, and let
$$\mathcal{M}^V_v = v^\perp \cap \mathcal{M}^V.$$
Let $\Gamma$ denote the group of orthogonal transformations of the
lattice $V$, and let
$$\Gamma_v \subset \Gamma$$
be the subgroup fixing $v$.
The moduli space of quasi-polarized $K3$ surfaces of degree $l$
is the quotient
$$\mathcal{M}_l = \mathcal{M}^V_v/ \Gamma_v.$$
The moduli space is a nonsingular orbifold. We refer the reader to
\cite{dolga} for a more detailed discussion.

In case $\bigtriangleup_l(h,d)\neq 0$,  the above 
construction of $\mathcal{M}_l$ shows the definitions of the
Noether-Lefschetz number by \eqref{def11} and \eqref{yyy12} agree.

\section{Gromov-Witten theory}\label{gwrev}
\subsection{BPS states for 3-folds}
Let $(X,L,\pi)$ be a quasi-polarized family of $K3$ surfaces.
While $X$ may not be a projective variety,
$X$ carries a $(1,1)$-form $\omega_K$ which is K\"ahler on the
$K3$ fibers of $\pi$. The existence of a fiberwise K\"ahler
form is sufficient to define Gromov-Witten theory for
vertical classes  
$$0\neq \gamma \in H_2(X,\mathbb{Z})^\pi.$$
The fiberwise K\"ahler form $\omega_K$ is obtained by a small
perturbation of the quasi-K\"ahler form obtained from the
quasi-polarization. The associated Gromov-Witten theory is
independent of the perturbation used.{\footnote{See
\cite{jlee,peng} for treatments of Gromov-Witten invariants
for fiberwise K\"ahler geometry.}}
 
Let $\overline{M}_{g}(X,\gamma)$ be the moduli space of
stable maps from connected genus $g$ curves to $X$.
Gromov-Witten theory is defined
by
integration against the virtual class,
\begin{equation}
\label{klk}
N_{g,\gamma}^X
 = \int_{[\overline{M}_{g}(X,\gamma)]^{vir}} 1.
\end{equation}
The expected dimension of the moduli space is 0.

The Gromov-Witten potential $F^X(\lambda,v)$ for nonzero vertical classes
is the series
$${F}^X=
\sum_{g\geq 0}\   \sum_{0\neq \gamma\in H_2(X,\mathbb{Z})^\pi}  
 N^X_{g,\gamma} \ \lambda^{2g-2} v^\gamma$$
where $\lambda$ and $v$ are the genus and curve class variables.
The
BPS counts $n_{g,\gamma}^X$
of Gopakumar and Vafa are uniquely defined 
by the following equation:
\begin{equation*}
F^X  =   \sum_{g\geq 0} \  \sum_{0\neq \gamma\in
H_2(X,\mathbb{Z})^\pi} 
 n_{g,\gamma}^X \ \lambda^{2g-2} \sum_{d>0}
\frac{1}{d}\left( \frac{\sin
({d\lambda/2})}{\lambda/2}\right)^{2g-2} v^{d\gamma}. 
\end{equation*}
Conjecturally, the invariants $n_{g,\gamma}^X$ are integral and
obtained from the cohomology of an as yet unspecified moduli
space of sheaves on $X$.

\subsection{Reduced theory}
\label{redth}

Let $C$ be a connected, nodal, genus $g$ curve. Let $S$
be a $K3$ surface, and let $\beta\in \text{Pic}(S)$
be a nonzero class.
The moduli space
$M_C(S,\beta)$ parameterizes maps from $C$ to 
$S$ of class $\beta$.
Let 
$$\nu: C \times M_C(S,\beta) \rightarrow M_C(S,\beta)$$
denote the projection, and let
$$f:  C \times M_C(S,\beta) \rightarrow S$$
denote the universal map.
The canonical morphism
\begin{equation}\label{reld}
R^\bullet \nu_*(f^* S)^\vee \rightarrow L_{M_C}^\bullet
\end{equation}
determines a perfect obstruction theory on $M_C(S,\beta)$, 
see \cite{Beh,BehFan,LiTian}.
Here, $L_{M_C}^\bullet$ denotes the cotangent complex of $M_C(S,\beta)$.

Let $\Omega_S$ denote the cotangent bundle of $S$.
Let $\Omega_\nu$ and $\omega_\nu$ denote respectively the sheaf
of relative differentials of $\nu$ and the relative dualizing sheaf of
$\nu$.
There are  canonical maps
\begin{equation}
f^*(\Omega_{S}) \rightarrow \Omega_\nu \rightarrow \omega_\nu
\end{equation}
The sections of the canonical bundle $H^0(S,K_S)$ determine
a 1-dimensional space of
holomorphic symplectic forms. Hence, there is a canonical
isomorphism
$$T_S \otimes H^0(S,K_S) \stackrel{\sim}{=} \Omega_S$$
where $T_S$ is the tangent bundle.
We obtain
a map
$$f^*(T_{S}) \rightarrow \omega_\nu \otimes (H^0(S,K_S))^\vee$$
and a map
\begin{equation}\label{c2r}
R^\bullet\nu_*(\omega_\nu)^\vee \otimes H^0(S,K_S)
 \rightarrow R^\bullet\nu_*(f^*T_{S})^\vee.
\end{equation}

From \eqref{c2r}, we obtain the cut-off map
$$\iota: \tau_{\leq -1} R^\bullet\nu_*(\omega_\nu)^\vee \otimes H^0(S,K_S)
  \rightarrow
R^\bullet\nu_*(f^*T_{S})^\vee.$$
The complex $\tau_{\leq -1} R^\bullet\nu_*(\omega_\nu)^\vee\otimes H^0(S,K_S)$ 
is represented
by a trivial bundle of rank $1$ tensored with
$H^0(S,K_S)$ in degree $-1$.
Consider the mapping cone $C(\iota)$ of $\iota$.
Certainly $R^\bullet\pi_*(f^*T_{S})^\vee$ is represented
by a two term complex. An elementary argument using nonvanishing $\beta\neq 0$
shows
the complex $C(\iota)$ is also two term.

By Ran's results{\footnote{The required deformation theory can also
be found in a recent paper by M. Manetti \cite{Man}. A different approach
to the construction of the reduced virtual class is available in \cite{ared}.
}} on deformation theory and the semiregularity map, 
there is a canonical map
\begin{equation}
\label{rvc}
C(\iota) \rightarrow L_{M_C}^\bullet
\end{equation}
induced by \eqref{reld}, see \cite{zzz2}.
Ran proves the obstructions to deforming maps from $C$ to
a holomorphic symplectic manifold lie in the kernel
of the semiregularity map. After dualizing,  Ran's result
precisely shows \eqref{reld} factors through the cone $C(\iota)$.

The map \eqref{rvc} defines a {\em new} perfect obstruction theory
on $M_C(S,\beta)$. The conditions of cohomology isomorphism in degree 0 and
the cohomology surjectivity in degree $-1$ are both induced from
the perfect obstruction theory \eqref{reld}. 
We view \eqref{reld} as the {\em standard} obstruction theory and
\eqref{rvc} as the {\em reduced} obstruction theory.

Following \cite{Beh, BehFan}, the morphism
 \eqref{rvc} is an obstruction theory
of maps to $S$ 
relative to the Artin stack
${\mathfrak M}_g$ of genus $g$ curves. 
A reduced absolute obstruction theory 
\begin{equation}\label{yoyoyo}
E^\bullet \rightarrow L_{\overline{M}_g(S,\beta)}^\bullet
\end{equation}
is
obtained via a distinguished triangle in the usual way, see 
\cite{Beh,BehFan,LiTian}.
The obstruction theory \eqref{yoyoyo} yields a reduced virtual
class
$$[\overline{M}_g(S,\beta)]^{red} \in A_{g}(\overline{M}_g(S,\beta),
\mathbb{Q})$$
of dimension $g$.

\subsection{BPS for $K3$ surfaces}\label{bpsk3}
Let $(S,\omega_K)$ be a $K3$ surface with a K\"ahler form $\omega_K$.
Let $\beta \in \text{Pic}(S)$ be a nonzero class of positive
degree 
$$\int_\beta \omega_K >0.$$
 We are interested
in the following reduced Gromov-Witten integrals,
\begin{equation}\label{d432}
R_{g,\beta}=\int_{[\overline{M}_g(S,\beta)]^{red}}
(-1)^g \lambda_g.
\end{equation}
Here, the integrand $\lambda_g$ is the top Chern
class of the Hodge bundle 
$$\mathbb{E}_g \rarr \overline{M}_g(S,\beta)$$
with fiber $H^0(C,\omega_C)$ over moduli point
$$[f:C\rarr S]\in \overline{M}_g(S,\beta).$$
See \cite{FP,GP} for a discussion of Hodge classes in
Gromov-Witten theory.

The definition of the BPS counts associated to the Hodge integrals
\eqref{d432} is straightforward. Let 
$\alpha\in \text{Pic}(S)$ be a primitive class of
positive degree with respect to $\omega_K$.
The Gromov-Witten potential $F_{{\alpha}}(\lambda,v)$ 
for classes proportional
to ${\alpha}$
is 
$${F}_{{\alpha}}=
\sum_{g\geq 0}\   \sum_{m>0}\   R_{g,m\alpha} \ \lambda^{2g-2} 
v^{m{\alpha}}.$$
The
BPS counts $r_{g,m\alpha}$ are uniquely defined 
by the following equation:
\begin{equation*}
F_\alpha =   \ \ \ \sum_{g\geq 0}  \ \sum_{m>0} \
 r_{g,m\alpha} \ \lambda^{2g-2} \sum_{d>0}
\frac{1}{d}\left( \frac{\sin
({d\lambda/2})}{\lambda/2}\right)^{2g-2} v^{dm\alpha}. 
\end{equation*}
We have defined BPS counts  
for both primitive and
divisible classes.

The string theoretic calculations of Katz, Klemm and Vafa \cite{kkv}
via heterotic duality yield two conjectures.

\begin{Conjecture} \label{xxx1}
The BPS count $r_{g,\beta}$ depends upon $\beta$ only through the square $\int_S \beta^2$.
\end{Conjecture}

Assuming the validity of
Conjecture \ref{xxx1}, let $r_{g,h}$ denote the BPS count associated
to a class $\beta$ satisfying
$$\int_S \beta^2 = 2h-2.$$
Conjecture \ref{xxx1} is rather surprising from the point
of view of Gromov-Witten theory. By deformation
arguments, the invariants $R_{g,\beta}$
depend upon both the divisibility $m$ of $\beta$ and $\int_S \beta^2$.
Hence, BPS counts $r_{g,m,h}$ depending upon both the divisibility and
the norm are well-defined unconditionally.

\begin{Conjecture} \label{xxx2}
The BPS counts $r_{g,h}$ are uniquely determined by the
following equation:
$$\sum_{g\geq 0} \sum_{h\geq 0} (-1)^g r_{g,h}(y^{\frac{1}{2}} - y^{-\frac{1}{2}})^{2g}q^h =
\prod_{n\geq 1} \frac{1}{(1-q^n)^{20} (1-yq^n)^2 (1-y^{-1}q^n)^2}.$$
\end{Conjecture}

As a consequence of Conjecture 2, $r_{g,h}$ vanishes if $g>h$ and
$$r_{g,g}=(-1)^g (g+1).$$
The first values are tabulated below:

\vspace{18pt}

\begin{tabular}{|c||ccccc|}
        \hline
\textbf{}
$r_{g,h}$&    $h= 0$ & 1  & 2 & 3 & 4 \\
        \hline \hline
$g=0$ & $1$ & $24$ & $324$ & 
$3200$ &$25650$  \\
1      &  & $-2$ & 
$-54$ & $-800$  & $-8550$      \\
2      & & & $3$ & 
$88$ & $1401$       \\
3      & &  & 
 & $-4$  & $-126$       \\
4      &  &  & 
 &   & 5       \\
       \hline
\end{tabular}

\vspace{18pt}

The right side Conjecture \ref{xxx2} is related to the 
generating series of Hodge numbers of the Hilbert schemes of points
$\text{Hilb}(S,n)$.
The genus 0 specialization of Conjecture \ref{xxx2} recovers the
Yau-Zaslow formula
$$\sum_{h\geq 0} r_{0,h} q^h = \prod_{n\geq 1} \frac{1}{(1-q^n)^{24}}$$
related to the Euler characteristics of $\text{Hilb}(S,n)$.

The Conjectures are proven in very few cases.
A mathematical approach to the genus 0 Yau-Zaslow formula 
following \cite{yauz}
can be found in \cite{beu,xc,fgd}.
The Yau-Zaslow formula is proven for
primitive classes $\beta$ by Bryan and Leung \cite{brl}. 
If $\beta$ has divisibility 2, 
the Yau-Zaslow formula is proven 
by Lee and Leung 
 in \cite{ll}. Using Theorem 1, a complete proof of the
Yau-Zaslow formula for all divisibilities is given in 
\cite{KMPS}. Since
$$R_{1,\beta} = \int_{[\overline{M}_1(S,\beta)]^{red}} -\lambda_1
 =-\frac{\lan \beta, \beta \ran}{24} R_{0,\beta},$$
we obtain 
$$r_{1,h} = -\frac{h}{12}\  r_{0,h}$$
and Conjectures 1 and 2 for genus 1
from the genus 0 results.

Conjecture 2 for primitive classes $\beta$ is connected
to Euler characteristics of the moduli spaces of
stable pairs on $K3$ by the correspondence of \cite{PT1,PT2}.
A proof of Conjecture 2 for primitive classes is given in \cite{MPT}.


\section{Theorem 1} \label{tmm}

\subsection{Result} 
Consider  a quasi-polarized family of $K3$ surfaces of degree $l$
as in Section \ref{quasp},
$$\pi:X \rarr C\ .$$
 We restate Theorem 1 in terms
of $\gamma \in H_2(X,\mathbb{Z})^\pi$ following the
notation of Section \ref{lsy}.

\setcounter{Theorem}{0}
\begin{Theorem}
 For $\gamma\neq 0$,
$$n_{g,\gamma}^X= \sum_{h} \sum_{m=1}^{\infty}
r_{g,m,h}\cdot  NL_{m,h,\gamma}^\pi.$$
\end{Theorem}

\subsection{Proof}
Since the formulas relating the BPS counts to Gromov-Witten invariants
are the same for $X$ and the $K3$ surface,  Theorem 1 is 
equivalent to the analogous 
Gromov-Witten statement:
\begin{equation}\label{cq29}
N_{g,\gamma}^X= \sum_{h} \sum_{m=1}^{\infty}
R_{g,m,h}\cdot  NL_{m,h,\gamma}^\pi
\end{equation}
for $\gamma\neq 0$.

Following the notation of Section \ref{nnn}, let $\sigma$ denote
the section
$$\sigma: C \rarr \mathcal{M}^{\mathcal{V}}$$
determined by the Hodge structure of the $K3$ fibers
$$\sigma(\xi)= [H^0(X,K_{X_\xi})]\in \mathcal{M}^{\mathcal{V}_\xi}.$$
For each $\xi\in C$, let
$$\mathcal{V}_\xi(m,h,\gamma) \subset \mathcal{V}_\xi $$
be the set of classes with divisibility $m$, square $2h-2$,
and push-forward $\gamma$. Let
$$B_\xi(m,h,\gamma)= \{\ \beta \in \mathcal{V}_\xi(m,h,\gamma)
 \ | \ \sigma(\xi)
\in \beta^\perp\ \}.$$
By Proposition 1, the set $B_\xi(m,h,\gamma)$ is
finite.

Equation \eqref{cq29} is proven by showing the contributions of the 
classes $B_\xi(m,h,\gamma)$ to both sides are the same.
The set
$$B(m,h,\gamma) = \bigcup B_\xi(m,h,\gamma) \subset \mathcal{V}$$
can be divided into two disjoint subsets
$$B(m,h,\gamma) = B_{\text{iso}}(m,h,\gamma) \cup B_{\infty}(m,h,\gamma).$$
The elements of $B_{\text{iso}}(m,h,\gamma)$ are isolated while 
the elements of $B_{\infty}(m,h,\gamma)$ form a finite local system over
$C$,
\begin{equation}\label{nq5}
\epsilon:B_{\infty}(m,h,\gamma)\rarr C.
\end{equation}
We address the contributions of the isolated issues and the
local system separately.

Consider first the local system \eqref{nq5}. The contribution of
$\epsilon$ to the Gromov-Witten invariant $N_{g,\gamma}^X$ is
the integral
$$N_{g,\epsilon}^X=\int_{[\overline{M}_{g}(X,\epsilon)]^{vir}} 1$$
where $\overline{M}_g(X,\epsilon)\subset \overline{M}_g(X,\gamma)$ 
is the connected component{\footnote{By connected component,
we mean both open and closed. Formally, the condition
is usually stated as a union of connected components.}}
of the moduli space of stable maps which represent curve classes
in $\epsilon$. Alternatively, 
\begin{equation}\label{stan}
N_{g,\epsilon}^X = \int_{[\overline{M}_g(\pi,\epsilon)]^{vir}} 
c_g(\mathbb{E}^*_g \otimes T_C)
\end{equation}
where 
$\overline{M}_g(\pi,\epsilon)\subset \overline{M}_g(\pi,\gamma)$ 
is a 
connected component 
of the relative moduli space of maps.
By standard arguments \cite{FP}, the difference in the
absolute and relative obstruction theories is $\mathbb{E}^*_g \otimes T_C$
and hence yields the
Hodge integrand in \eqref{stan}.

The family $\pi$ determines a canonical line bundle
$$K \rarr C$$
with fiber $H^0(X_\xi, K_{X_\xi})$ over $\xi\in C$.
By the construction of the reduced class in Section \ref{redth},
$$[\overline{M}_g(\pi,\epsilon)]^{vir} = 
c_1(K^*) \cap [\overline{M}_g(\pi,\epsilon)]^{red}$$
where, on the right side, the reduced virtual
class for the relative moduli space of maps
appears.
Expanding \eqref{stan} yields
\begin{eqnarray*}
N_{g,\epsilon}^X & = & \int_{[\overline{M}_g(\pi,\epsilon)]^{red}} 
c_g(\mathbb{E}^*_g \otimes T_C)\cdot c_1(K^*) \\
& = &
 \int_{[\overline{M}_g(K3,m\alpha)]^{red}} 
(-1)^g\lambda_g \cdot \int_{B_\infty(m,h,\gamma)} c_1(K^*) \\
& = &  R_{g,m,h} \cdot \int_{B_\infty(m,h,\gamma)} c_1(K^*).
\end{eqnarray*}
In the second equality, $\alpha$ is primitive and
satisfies $$\langle m\alpha, m\alpha \rangle=2h-2.$$

The contribution of the local system $\epsilon$ to the
Noether-Lefschetz number $NL_{m,h,\gamma}^\pi$ is much
easier to calculate. The local system represents an
excess intersection contribution 
$$\int_{B_\infty(m,h,\gamma)} c_1(\text{Norm})$$
where $\text{Norm}$ is the line bundle with fiber
$$\text{Hom}(H^0(X_\xi, K_{X_\xi}), \com \cdot \beta)$$
at $\beta \in B_\infty(m,h,\gamma)$ lying over
$\xi\in C$. Over $B_\infty(m,h,\gamma)$, the fibration
$\com\cdot \beta$ is a trivial line bundle. 
Hence, the excess contribution of
$B_\infty(m,h,\gamma)$ to $NL^\pi_{m,h,\gamma}$ is 
$$\int_{B_\infty(m,h,\gamma)} c_1(K^*).$$
We conclude the contributions of $B_\infty(m,h,\gamma)$ to
the left and right sides of equation \eqref{cq29} exactly match.

We consider now the contributions of the isolated classes
$B_{\text iso}(m,h,\gamma)$ to the two sides of
\eqref{cq29}. Let
$$\beta \in B_{\text{ iso}}(m,h,\gamma)$$
be an isolated class lying over $\xi\in C$. We
trivialize  $\mathcal{M}^{\mathcal{V}}$ over a Euclidean
open set $U\subset C$ as in Section \ref{nnn}.
The local intersection of the section $\sigma$ with the divisor 
$$D_\beta^{V_\xi} \times U \subset M^{V_\xi} \times U$$
has an isolated point corresponding to
$(\beta,\xi)$. The local intersection multiplicity may not
be 1. However, by deformation equivalence of the
Gromov-Witten contributions on the left side of \eqref{cq29} and the
intersection products on the right side of \eqref{cq29},
we may assume {\em the local intersection multiplicity
is 1} after local holomorphic perturbation of the section $\sigma$.
Then, the contribution of the isolated class $\beta$ to
$NL^\pi_{m,h,\gamma}$ is certainly 1.

The final step is to show the contribution of the isolated
class $\beta$ with intersection multiplicity 1 to $N^X_{g,\gamma}$
is simply $R_{g,m,h}$. The result is obtained by a comparison
of obstruction theories. 

By the multiplicity $1$ hypothesis, a connected component of
the moduli space of stable maps to $X$  coincides with the
 moduli stable of stable maps to fiber $X_\xi$, 
\begin{equation}\label{ffeww}
\overline{M}_{g}(X_\xi,\beta) \subset
\overline{M}_g(X,\gamma).
\end{equation}
At the level of points, the assertion is obvious.
The multiplicity 1 conditions prohibits any infinitesimal
deformations of maps away from the fiber $X_\xi$ and
implies the scheme theoretic assertion.

From the fibration $\pi$, we obtain an exact sequence
\begin{equation}\label{long}
0\rarr T_{X_\xi} \rarr T_X|_{X_\xi} \rarr T_{C,\xi} \rarr 0,
\end{equation}
and an induced map
$$ \widetilde{\iota}:R^\bullet \nu_*(f^*T_{X_\xi})^\vee \rarr T_{C,\xi}^*$$
where the second complex is a trivial bundle in degree $-1$.
Following the notation of
Section \ref{redth}, we have a canonical map
$$\iota: H^0(X_\xi,K_{X_\xi})
  \rightarrow
R^\bullet\nu_*(f^*T_{X_\xi})^\vee$$
where the first complex is
a trivial bundle with fiber
$H^0(X_\xi,K_{X_\xi})$ in degree $-1$.
By Lemma \ref{fafa} below, the composition 
$$ \widetilde{\iota} \circ \iota: H^0(X_\xi,K_{X_\xi}) \rarr
T_{C,\xi}^*$$
is an isomorphism. Hence,
by  sequence \eqref{long},
the obstruction theories 
$R^\bullet\nu_*(f^*T_{X})^\vee$ and $C(\iota)$
differ by only by the Hodge bundle $\mathbb{E}_g \otimes T_{C,\xi}^*$.
We conclude
$$[\overline{M}_{g}(X_\xi,\beta)]^{vir_X} =
(-1)^g \lambda_g \cap [\overline{M}_{g}(X_\xi,\beta)]^{red}$$
where the virtual class on the left is obtained
from the obstruction theory of maps to $X$ via \eqref{ffeww}.
The contribution of the isolated class $\beta$ to
$N_{g,\gamma}^X$ is thus $R_{g,h,m}$.

Since the contributions of $B_{\text{iso}}(m,h,\gamma)$ to
the left and right sides of equation \eqref{cq29} also match,
the proof of Theorem 1 is complete. \qed

\begin{Lemma} \label{fafa} The composition
$$ \widetilde{\iota} \circ \iota: H^0(X_\xi,K_{X_\xi}) \rarr
T_{C,\xi}^*$$
is an isomorphism.
\end{Lemma}
\bpf
Consider the differential of the period map at $\xi$,
$$T_{C,\xi} \rightarrow H^1(T_{X_{\xi}}) \rightarrow \mathrm{Hom}(H^0(K_{X_{\xi}}),H^{1}(\Omega_{X_{\xi}})).$$
The multiplicity $1$ condition implies that the image of this map is not contained
in the tangent space to the hyperplane $\beta^{\perp} = 0$.  More explicitly, if we apply
the cup-product pairing of $H^1(\Omega_{X_\xi})$ with the class 
$\beta \in H^{2}(X_{\xi},\mathbb{Z})$, the composition 
$$T_{C,\xi}\rarr H^{0}(K_{X_{\xi}})^{\ast}\otimes H^{1}(\Omega_{X_{\xi}})
\xrightarrow{\beta\cup} H^{0}(K_{X_{\xi}})^{\ast}\otimes \com$$  
is nonzero.
This sequence can be included in the diagram
$$\xymatrix{
T_{C_{\xi}}\ar[r]\ar@{=}[d] & H^{1}(T_{X_{\xi}})\ar[r]\ar[d] & H^{0}(K_{X_\xi})^{\ast}\otimes H^{1}(\Omega_{X_\xi})\ar[r]^-{\beta\cup}\ar[d] & H^{0}(K_{X_\xi})^{\ast}\ar@{=}[d]\\
T_{C_{\xi}}\ar[r] & R^{\bullet}\nu_{\ast}(f^{\ast}T_{X_\xi})\ar[r] & H^{0}(K_{X_\xi})^{\ast}\otimes R^{\bullet}\nu_{\ast}(f^{\ast}\Omega_{X_\xi})\ar[r]&H^{0}(K_{X_\xi})^{\ast}}
$$
where the vertical maps are given by base-change morphisms and the bottom row
is the map $(\widetilde{\iota}\circ \iota)^{\ast}$.
Standard comparison results imply that this diagram
commutes.  Since the top row is nonvanishing, so is the bottom row.
\epf

\subsection{Conjectures 1 and 2 revisited}
The proof of Conjectures 1 and 2 in 
the following case allows us to bound from below the
$h$ summation in Theorem 1.

\begin{Lemma}\label{aaas} If $\int_{K3} \beta^2 <0$, then 
$r_{g,\beta}= 1$
if 
$$g=0 \ \text{ and }\  \int_{K3} \beta^2 = -2$$
 and $r_{g,\beta}=0$ otherwise.
\end{Lemma}

\begin{proof}
Let $S$ be a $K3$ surface, and let 
$\beta\in \text{Pic}(S)$ be primitive with 
$$\int_S \beta^2 =-2.$$ We may
assume $\beta$ is represented by an isolated $-2$ curve $P\subset S$.
Let 
$$\pi: X \rarr \bigtriangleup_0$$
be a 1-parameter deformation of $S$ over the disk $\bigtriangleup_0$
for which $\beta$ 
fails (even infinitesimally) to remain algebraic. By the
proof of Theorem 1, the reduced
invariants $r_{g,m,\beta}$ are obtained{\footnote{The local NL intersection
number here is 1.}} from the
contribution of $P$ to the BPS state counts of $X$.
Since $P$ is a rigid $(-1,-1)$ curve, $P$
contributes a single BPS state \cite{FP}. We conclude
$$r_{g,m,\beta}=1$$
if $(g,m)=(0,1)$ and $r_{g,m,\beta}=0$ otherwise.

If $\beta\in \text{Pic}(S)$ 
is primitive with square $2h-2$ strictly less than $-2$, then
all reduced invariants $r_{g,m,\beta}$ vanish. The proof is obtained by
considering elliptically fibered $K3$ surfaces
$S \rarr \proj^1$.
Let 
$$[s],[f]\in \text{Pic}(S)$$ be the classes of a section and
a fiber respectively.
Then,
$$[s]+h[f], \ -[s]-h[f] \in \text{Pic}(S)$$
are both primitive with square $2h-2$.
Since the moduli spaces
$$\overline{M}_{g}\left(S,m([s]+h[f])\right), \
 \overline{M}_{g}\left(S,m(-[s]-h[f])\right)$$
are easily seen to be empty, all reduced invariants $r_{g,m,\beta}$
vanish.
\end{proof}

By Lemma \ref{aaas}, the integrals $r_{g,m,h<0}$ all vanish. Hence,
Theorem 1 may be written as
$$n_{g,\gamma}^X= \sum_{h\geq 0} \sum_{m=1}^{\infty}
r_{g,m,h}\cdot  NL_{m,h,\gamma}^\pi.$$
If Conjecture 1 and the vanishing $r_{g,h}$ for $g>h$ of Conjecture
2 hold, then
$$r_{g,h}=r_{g,m,h}$$
and
Theorem 1 implies  the following result.
by relation \eqref{zz23}.

\vspace{10pt}
\noindent{\bf Theorem $\mathbf{1^*.}$}
{\em
For $\gamma\neq 0$,
$$n_{g,\gamma}^X= \sum_{h\geq g} 
r_{g,h}\cdot  NL_{h,\gamma}^\pi\ .$$
}
\vspace{10pt}

The asterisk here indicates the  dependence of Theorem $1^*$
 upon Conjectures
1 and 2.

\subsection{Invertibility}
Theorem $1^*$ and Conjecture 2
imply the BPS states $n_{g,\gamma}^X$ of the total space contain
exactly the same information as the Noether-Lefschetz numbers
$NL^\pi_{h,\gamma}$.

\vspace{10pt}
\noindent{\bf Proposition $\mathbf{4^*.}$}
{\em
For $\gamma\in H_2(X,\mathbb{Z})^\pi$
of positive degree,
the invariants $\{n_{g,\gamma}(\pi)\}_{g\geq 0}$ determine
the Noether-Lefschetz numbers $\{NL_{h,\gamma}(\pi)\}_{h\geq 0}$
in terms of the invariants $\{r_{g,h}\}_{g,h\geq 0}$.}
\vspace{10pt}

\bpf 
Fix $\gamma \in H_2(X,\mathbb{Z})^\pi$.
By Proposition 2, the numbers $NL_{h,\gamma}(\pi)$
vanish for $h>h_{top}$. So we need only determine
$$NL_{0,\gamma} , \ldots NL_{h_{top},\gamma}.$$
The equations
$$n_{g,\gamma}(\pi)= \sum_{h=g}^{h_{top}} 
r_{g,h}\cdot  NL_{h,\gamma}(\pi)$$
for $g=0, \ldots, h_{top}$ of Theorem $1^*$
are triangular and invertible by Conjecture 2.
\epf

\section{Modular forms}\label{modularforms}
\subsection{Overview}

We explain here Borcherds' work \cite{borch} relating Noether-Lefschetz 
numbers to Fourier coefficients of modular 
forms.{\footnote{Borcherds' original result is modular only up to a
$\text{Gal}(\overline{\mathbb{Q}}/\mathbb{Q})$-action. The
strengthening of \cite{borch} by the
more recent rationality result of
\cite{mcgraw} removes the $\text{Gal}(\overline{\mathbb{Q}}/\mathbb{Q})$
issue.}}
His results apply in great generality 
to arithmetic quotients of symmetric spaces associated to the orthogonal group 
$O(2,n)$ for any $n$.  While we are mainly interested 
in the case of $O(2,19)$, we will first explain 
the statement in full generality.  Other values of $n$ play a role, for
example, 
in studying 1-parameter families of $K3$ surfaces with 
generic Picard rank at least $2$.

\subsection{Vector-valued modular forms of half-integral weight}

We first summarize standard facts and notation regarding modular forms of 
half-integral weight.  In order to make sense of the modular transformation 
law with half-integer exponents, 
a double cover of the standard modular group $SL_{2}(\Z)$ is required. 

The metaplectic group $Mp_{2}(\R)$ is the 
unique connected double cover of $SL_{2}(\R)$.  The elements of 
$Mp_{2}(\R)$
 can be written in the form
$$\left(\left(\begin{array}{cc} a & b\\
c & d\end{array}\right), \phi(\tau) = \pm\sqrt{c\tau+d}\right)$$
where $\left(\begin{array}{cc}a & b \\ c & d\end{array}\right) \in SL_{2}(\R)$
and $\phi(\tau)$ is a choice of square root of the function $c\tau+d$ on the upper-half plane $\mathcal{H}$.
The group structure is defined by the product 
$$\left(A_{1},\phi_{1}(\tau)\right)\cdot\left(A_{2},\phi_{2}(\tau)\right) 
= \left(A_{1}A_{2}, \phi_{1}(A_{2}\tau)\phi_{2}(\tau)\right).$$  
Here, we write $A\tau$ for the usual action of $SL_{2}(\R)$ on $\tau\in\mathcal{H}$.

The group $Mp_{2}(\Z)$ is the preimage of $SL_{2}(\Z)$ under the projection map  
$$\pi: Mp_{2}(\R) \rightarrow SL_{2}(\R).$$
It is generated by the two elements
$$T = \left(\left(\begin{array}{cc} 1 & 1\\ 0 & 1\end{array}\right), 1\right),
S =  \left(\left(\begin{array}{cc} 0 & -1\\ 1 & 0\end{array}\right), \sqrt{\tau}\right),$$
where $\sqrt{\tau}$ denotes the choice of square root with positive real part.

Suppose we are given a representation 
$\rho$ of $Mp_{2}(\Z)$ on a finite-dimensional 
complex vector space $V$ with the property that $\rho$ factors 
through a finite quotient.  Given $k \in \frac{1}{2}\Z$,
we define a modular form of weight $k$ and type $\rho$ to be a holomorphic function
$$f: \mathcal{H} \rightarrow V$$
such that, for all 
$g =\left(A, \phi(\tau)\right)\in Mp_{2}(\Z)$,
we have
$$f(A\tau) = \phi(\tau)^{2k}\cdot \rho(g)(f(\tau)).$$  For $k\in \Z$ and $\rho$ trivial, this 
reduces to the usual transformation rule.

If we fix an eigenbasis $\{v_{\gamma}\}$ for $V$ 
with respect to $T$, we can take the Fourier expansion 
of each component of $f$ at the cusp at infinity.  That is, we write
$$f(\tau) = \sum_{\gamma} \sum_{k\in \Z} c_{k,\gamma} q^{k/R} v_{\gamma} \in V$$ 
where 
$$q = e^{2\pi i \tau}$$
and
$R$ is the smallest positive integer for which $T^{R}\in \mathrm{Ker}(\rho)$.    
The function $f$ is holomorphic at infinity if $c_{k,r} = 0$
for $k < 0$.  The space $\mathrm{Mod}(Mp_{2}(\Z), k,\rho)$ of holomorphic modular forms of 
weight $k$ and type $\rho$ is finite-dimensional.

 Given an integral lattice $M$ with
an even bilinear form $\lan,\ran$ with signature $(2,n)$, we associate to $M$ the following unitary representation of $Mp_{2}(\Z)$.
Let $$M^{\vee} \subset M \otimes \Q$$
denote the dual lattice and $M^{\vee}/M$ the finite quotient.  
The pairing $\lan ,\ran$ extends linearly
to a $\Q$-valued pairing on $M^{\vee}$.  The functions 
$\frac{1}{2}\lan \gamma,\gamma\ran$ and $\lan \gamma,\delta\ran$ descend to $\Q/\Z$-valued
functions on $M^{\vee}/M$.

We construct a representation $\rho_{M}$ of $Mp_{2}(\Z)$ on the group algebra
$\com[M^{\vee}/M]$.  
It suffices to define $\rho_{M}$ in terms of the action of the generators $T$ and $S$ with respect
to the standard basis $v_{\gamma}$ for  $\gamma \in M^{\vee}/M$,

\begin{align*}
\rho_{M}(T)v_{\gamma} &= e^{2\pi i\frac{\lan\gamma,\gamma\ran}{2}} v_{\gamma}\ ,\\
\rho_{M}(S)v_{\gamma} &= \frac{\sqrt{i}^{n-2}}{\sqrt{|M^{\vee}/M|}}
\sum_{\delta} e^{-2\pi i \lan\gamma,\delta\ran} v_{\delta}\ .
\end{align*}

Let $N$ denote the smallest positive integer
for which $N\lan \gamma,\gamma\ran/2 \in \Z$ for all $\gamma \in M^{\vee}$.
The representation factors through a double cover of $SL_2(\mathbb{Z}/N\mathbb{Z})$. 
We will be primarily interested in
the dual representation $\rho_{M}^{\ast}$ of $Mp_{2}(\Z)$ on $\com[M^{\vee}/M]$.  We have 
given the action of $\rho_{M}$ to match Borcherds' notation.

\subsection{Heegner divisors}

Given the lattice $M$ of type $(2,n)$ as before, 
consider the Hermitian symmetric domain
$$\mathcal{D} = \left\{ \omega \in \proj(M \otimes_{\Z}\com)\ |\ \lan \omega,\omega\ran = 0, 
\lan \omega, \bar{\omega}\ran > 0 \right\}$$
naturally associated to $M$.   We will study the quotient
\begin{equation}\label{s352}
\mathcal{X}_{M} = \mathcal{D}/\Gamma_{M}
\end{equation}
of $\mathcal{D}$ by the arithmetic subgroup of $O(2,n)$
$$\Gamma_{M} = \left\{ g \in \mathrm{Aut}(M)\ |\ g \text{ acts trivially on } M^{\vee}/M\right\}.$$
The quotient \eqref{s352} is a quasi-projective algebraic variety.

For every $n\in\Q^{<0}$ and $\gamma \in M^{\vee}/M$, we associate a divisor class
$y_{n,\gamma} \in \mathrm{Pic}(\mathcal{X}_{M})$ as follows.  Given an element $v \in M^{\vee}$, there is an associated hyperplane
$$v^{\perp} = \left\{ \omega \in\mathcal{D}\ |\ \lan\omega,v\ran = 0\right\}.$$
Both $\lan v,v\ran$ and the residue class $v \bmod M$ are
invariant under the action of $\Gamma_M$.  Therefore, if we fix $n \in \Q$ 
and $\gamma \in M^{\vee}/M$, the set of $v \in M^{\vee}$ with 
$$\frac{1}{2}\lan v,v\ran = n,\ \  v \equiv \gamma \bmod M$$
is also $\Gamma_M$-invariant.  
The union over the set of the associated hyperplanes
$$\sum_{\shortstack{$\frac{1}{2}\lan v,v\ran = n$ \\$v \equiv \gamma \bmod M$}} v^{\perp}$$
is $\Gamma_M$-invariant and descends to an algebraic divisor
$$y_{n,\gamma}  = \left(\sum_{\frac{1}{2}\lan v,v\ran=n,\ v\equiv \gamma \bmod M} v^{\perp}\right)/\Gamma_{M}.$$
The $y_{n,\gamma}$ are the {\em Heegner divisors} of $\mathcal{X}_{M}$.
Because of the symmetry $v^{\perp} = (-v)^{\perp}$,
there is a redundancy $$y_{n,\gamma} = y_{n,-\gamma}$$ in our notation,
and $y_{n,\gamma}$ is multiplicity $2$ everywhere if $2\gamma \equiv 0 \bmod M$.  

In the degenerate case where $n = 0$, we have the following prescription.
The line bundle $\mathcal{O}(-1)$ on $\mathcal{D} \subset \proj(M\otimes_{\mathbb{Z}}\com)$
admits a natural $\Gamma_M$ action and therefore descends to a line bundle $K$ on 
$\mathcal{X}_{M}$.   If $n = 0$ and $\gamma = 0$, we set
$$y_{0,0} = K^{\ast}.$$
If $n=0$ and $\gamma \neq 0$, we set $y_{n,\gamma} = 0$.

We place the Heegner divisors 
in a formal power series $\Phi_{M}(q)$ with coefficients in
$\mathrm{Pic}(\mathcal{X}_M)\otimes \com[M^{\vee}/M]$.  More precisely, we consider the generating function 
$$\Phi(q) = \sum_{n\in \Q^{\geq 0}}\sum_{\gamma \in M^{\vee}/M} y_{-n,\gamma} q^{n} v_{\gamma} \in \mathrm{Pic}(\mathcal{X}_{M})[[q^{1/N}]]\otimes_{\Z} \com[M^{\vee}/M].$$

The main result of \cite{borch} together with the refinement of \cite{mcgraw} yield the following Theorem.

\vspace{10pt}
\noindent{\bf Theorem (\cite{borch},\cite{mcgraw})} {\em Let $M$ have signature $(2,n)$.
The generating function $\Phi(q)$ is an element of}
$$\mathrm{Pic}(\mathcal{X}_{M}) \otimes_{\Z} \mathrm{Mod}(Mp_{2}(\Z),1+\frac{n}{2},\rho_{M}^{\ast}).$$
\vspace{7pt}

As a consequence, given any linear functional
$$\lambda: \mathrm{Pic}(\mathcal{X}_{M})\otimes\com \rightarrow \com,$$
the contraction $\lambda(\Phi_{M}(q))$ is the Fourier expansion of a vector-valued
modular form of weight 
$1+\frac{n}{2}$ and type $\rho_{M}^{\ast}$.

Borcherds' proof uses the singular theta lift of \cite{borch2} to 
construct automorphic forms on $\mathcal{X}_{M}$ starting from vector-valued meromorphic modular forms on the upper half-plane.  
The zeroes and poles of these automorphic forms lie precisely along the Heegner
 divisors with multiplicity determined by the singular part of the initial modular form.  
Each such lifting gives a relation in $\mathrm{Pic}(\mathcal{X}_{M})$. The total collection of relations 
arising in this way are encoded in the modularity statement.

In \cite{borch2}, Borcherds only shows that $\Phi_{M}(q)$ lies in a 
certain Galois closure of the space of modular forms.  For the representations $\rho$ arising in \cite{borch2}, 
MacGraw proves in \cite{mcgraw} that  
$\mathrm{Mod}(Mp_{2}(\Z),k,\rho)$ admits a basis with rational coefficients. Therefore, the
Galois closure 
does not enlarge the space.

\subsection{Application to K3 surfaces}

Let $V$ be the rank 22 lattice obtained from the second cohomology of a $K3$ surface with fixed polarization $L$ of norm $l$.  
In order to apply Borcherds'
results to the moduli spaces $\mathcal{M}_{l}$, we consider the lattice
of signature $(2,19)$
 $$M = L^{\perp}= \left\{ v \in V \ |\ \lan L,v\ran = 0\right\}.$$   
A direct check yields
$$M \cong \Z w \oplus U^{2} \oplus E_{8}(-1)^{2}$$
 where $\lan w,w\ran = -l$.
Therefore $$M^{\vee}/M = \Z/l\Z$$
and is generated by $\frac{1}{l} w.$  Here, we will write $\rho_{l}$ for the representation
$\rho_{M}$.
 

From the definitions, we find
$\mathrm{Aut}(V, L) = \Gamma_{M}$, so
 we have the identification 
$$\mathcal{M}_{l} = \mathcal{X}_{M}.$$  
We claim the Heegner divisors correspond precisely to our Noether-Lefschetz divisors. 

\begin{Lemma} We have
$D_{h,d} = y_{n, \gamma}$,
where $$ n = -\frac{\Delta_{l}(h,d)}{2l} \ \text{ and }\  \gamma \equiv d (\frac{1}{l}w)\bmod M.$$
\end{Lemma}
\bpf
The Noether-Lefschetz divisor $D_{h,d}$ is the quotient by $\Gamma_{M}$
of the union of hyperplanes
$$\sum_{\shortstack{$\lan \beta,\beta\ran = 2h-2$\\ $\lan L,\beta\ran = d$}} \beta^{\perp}.$$
It therefore suffices to establish a bijection between the two sets of hyperplanes.
Given an element $ \beta \in V$ satisfying 
$$  \lan \beta,\beta\ran = 2h-2, \ \    \lan \beta,L\ran = d,$$ 
let $v  = \beta - \frac{d}{l}L \in M\otimes_{\mathbb{Z}} \Q$ be the projection of $\beta$ to $M = L^{\perp}$.
A direct calculation shows 
\begin{align*}
\frac{1}{2}\lan v,v\ran &= h-1 - \frac{d^{2}}{2l} = - \frac{\bigtriangleup_{l}(h,d)}{2l}\  ,\\
v &\equiv d\cdot (\frac{1}{l} w) \bmod M\ .
\end{align*}

Conversely, given $v \in M^{\vee}$ satisfying the above conditions,
$$\beta = v + \frac{d}{l} L$$ gives
the inverse construction.
Since $\beta^{\perp } = v^{\perp}$, we obtain the result.
\epf

It is important for our applications that the constant term $y_{0,0}$ of $\Phi_{M}(q)$ 
matches with the line bundle
$K^{\ast}$ from our excess calculation in the proof of Theorem 1. 
 This occurs because 
automorphic forms can be viewed as sections of powers of $K^{\ast}$ on $\mathcal{M}_{l}$.

Let $\pi$ be  a 1-parameter family   of 
quasi-polarized $K3$ surfaces of degree $l$, and let $\iota$ be
the associated morphism to moduli space:
$$\pi: X \rightarrow C,$$
$$\iota:  C \rightarrow \mathcal{M}_{l}.$$
We can apply Borcherds' theorem to the functional on $\mathrm{Pic}(\mathcal{M}_l)$ given by
$$D \mapsto \int_{C} \iota^{\ast}D.$$

\begin{Corollary} \label{333}
There is 
a vector-valued modular form of weight $21/2$ and type $\rho_{l}^{\ast}$,
$$\Phi^{\pi}(q) = \sum_{r=0}^{l-1} \Phi^{\pi}_{r}(q)v_{r} \in \com[[q^{1/2l}]]\otimes \com[\Z/l\Z],$$
with nonzero coefficients determined by the equality
$$NL^{\pi}_{h,d} = \Phi^{\pi}_{r}\left[ \frac{\bigtriangleup_{l}(h,d)}{2l}\right]$$
where $r \equiv d \bmod l$.
\end{Corollary}

\subsection{Quartic $K3$ surfaces}
\label{qqk3}
We now apply Borcherds' modularity to the study of $K3$ surfaces of degree $4$.  If 
$l=4$, the isomorphism class of a rank two lattice $(\mathbb{L},v)$ with primitive polarization 
$\lan v,v \ran =l$ is 
determined only by the discriminant $\bigtriangleup$.

Given a 1-parameter family $\pi: X \rightarrow C$ of quasi-polarized $K3$
surfaces of degree 4, we have the generating function
$$\Phi^{\pi}(q) = \Phi^{\pi}_{0}(q) v_{0} + \Phi^{\pi}_{1}(q) v_{1} + \Phi^{\pi}_{2}(q) v_{2} + \Phi^{\pi}_{3}(q) v_{3}$$
which is a modular form of weight $21/2$ and type $\rho^*_{4}$ by Corollary \ref{333}.

Consider the scalar-valued power series
$$\phi^{\pi}(q) = \Phi^{\pi}_{0}(q) + \frac{1}{2}\Phi^{\pi}_{1}(q) + \Phi^{\pi}_{2}(q) + \frac{1}{2}\Phi^{\pi}_{3}(q).$$
By chasing definitions, we see  $\phi^{\pi}(q)$ has the following property: 
\begin{equation}\label{zz3z}
NL^{\pi}_{h,d} = \phi^{\pi}\left[\frac{\bigtriangleup_{4}(h,d)}{8}\right].
\end{equation}
The factor of $1/2$ is included to correct for the redundancy $$\Phi^{\pi}_{1}(q) = \Phi^{\pi}_{3}(q).$$
\setcounter{Proposition}{4}

\begin{Proposition} The function
$\phi^{\pi}(q)$ is a homogeneous polynomial of degree $21$ in 
$$A= \sum_{n\in\Z} q^{\frac{n^{2}}{8}} \ \text{ and } \
B = \sum_{n \in \Z} (-1)^{n} q^{\frac{n^{2}}{8}}.$$
\end{Proposition}
\bpf
While the vector $\Phi^{\pi}(q)$ is modular with respect to the full metaplectic group, $\phi^{\pi}(q)$ 
is a priori only modular with respect to the subgroup $\widetilde{\Gamma}(8) = \mathrm{Ker}
(\rho^*_{4})$.
However, we can write $\phi^\pi(q)$ as a sum 
$$\phi^{\pi}(q) = \frac{3}{4}\phi_{+}(q) + \frac{1}{4}\phi_{-}(q)$$
where 
$$\phi_{+}(q) =  \Phi^{\pi}_{0}(q) + \Phi^{\pi}_{1}(q) +\Phi^{\pi}_{2}(q) + \Phi^{\pi}_{3}(q),$$  
$$\phi_{-}(q) =  \Phi^{\pi}_{0}(q) - \Phi^{\pi}_{1}(q) + \Phi^{\pi}_{2}(q) - \Phi^{\pi}_{3}(q).$$

Consider the congruence subgroup of $SL_{2}(\Z)$
$$\Gamma^{0}(8) = 
\left\{\left(\begin{array}{cc}a & b \\ c & d \end{array}\right)\in SL_{2}(\Z)\ |\ b \equiv 0 \bmod{8}\right\}.$$
A direct calculation of the representation $\rho^*_{4}$ shows 
that $\phi_{+}(q)$ and $\phi_{-}(q)$
are modular forms of weight $21/2$ with respect to 
$$\widetilde{\Gamma}^{0}(8) = \left\{ (A, \phi) \in Mp_{2}(\Z)\ |\ A \in \Gamma^{0}(8)\right\}$$
and distinct characters 
$$\chi_{+},\chi_{-}: \widetilde{\Gamma}^{0}(8) \rightarrow \com^{\ast}.$$
Moreover, $A$ and $B$ are modular forms of weight $1/2$ with respect
to $\widetilde{\Gamma}^{0}(8)$ and the same characters $\chi_{+}$ and $\chi_{-}$
respectively.

We will not describe $\chi_{\pm}$ explicitly.  While they are distinct, their squares are equal and 
$\chi = \chi_{+}^{2} = \chi_{-}^{2}$ descends to a character
$$\chi:\Gamma^{0}(8) \rightarrow \com^{\ast}.$$
The character $\chi$ is specified completely by the following evaluations:
$$\chi(\Gamma^{1}(8)) = 1,\ \chi\left(\begin{array}{cc} -1 & 0\\ 0 & -1\end{array}\right) = -1,\
\chi\left(\begin{array}{cc} 3 & 8\\ 1 & 3\end{array}\right) = -1$$
where 
$$\Gamma^{1}(8) = 
\left\{\left(\begin{array}{cc}a & b \\ c & d \end{array}\right)\in SL_{2}(\Z)\ |\
 b \equiv 0 \bmod{8}, a \equiv d \equiv 1 \bmod{8}\right\}.$$

Consider the space $\mathrm{Mod}(\Gamma^{0}(8), 11, \chi)$ of holomorphic modular forms of 
weight $11$ and type $\chi$.  The space $\mathrm{Mod}(\Gamma^{0}(8), 11, \chi)$
is  $12$-dimensional space with basis
$$A^{22}, A^{20}B^{2}, \cdots, A^{2}B^{20}, B^{22}.$$
Both
$\phi_{+}(q)\cdot A$ and $\phi_{-}(q)\cdot B$ lie in $\mathrm{Mod}(\Gamma^{0}(8),11,\chi)$.
Since $A^{22}/B$ and $B^{22}/A$ are not holomorphic at the boundary, we conclude $\phi_{\pm}(q)$
are each homogeneous polynomials of degree $21$ in $A$ and $B$ and therefore so is $\phi^{\pi}(q)$.
\epf

\section{Lefschetz pencil of quartics}\label{quarticcalc}
\subsection{Quartics} \label{quar}
A general Lefschetz pencil of quartics can be viewed as a hypersurface of type $(4,1)$,
\begin{equation}\label{f4563}
\pi: X_{4,1} \subset \proj^3 \times \proj^1 \rarr \proj^1
\end{equation}
where the last projection is onto the second factor.
Unfortunately, $\pi$ contains 108 nodal fibers, so the family \eqref{f4563}
does not fit the specifications of Section \ref{quasp}.

A family of quasi-polarized $K3$ surfaces of degree 4 can be obtained from the
Lefschetz pencil $\pi$ by the following construction.
Let 
\begin{equation}\label{f456}
\epsilon: C_{53} \stackrel{2-1}{\longrightarrow} \proj^1
\end{equation}
be the genus 53 hyperelliptic curve branched over the 108 points of $\proj^1$
corresponding to the nodal fibers of $\pi$.
The family
$$\epsilon^*(X_{4,1}) \rarr C_{53}$$
has 3-fold double point singularities over the 108 nodes of the fibers of the original family $\pi$.
Let
$$\widetilde{\pi}: \widetilde{X} \rarr C_{53}$$
be obtained from a small resolution 
$$\widetilde{X} \rarr \epsilon^*(X_{4,1}).$$
Then, $\widetilde{\pi}$ is easily seen to be a family of quasi-polarized $K3$ surfaces
of degree 4. The quasi-polarization is the pull-back of $\oh_{\proj^3}(1)$.

\subsection{Invariants}

The Noether-Lefschetz numbers are defined in Section \ref{nlnums} only for the family $\widetilde{\pi}$.
However, for convenience, we define
$$NL^{\pi}_{g,d}= \frac{1}{2}NL^{\widetilde{\pi}}_{g,d}\ \ .$$
Instead of a curve class $\gamma$, the degree $d$ against the polarization
is taken as the second subscript.

The family $\widetilde{\pi}$ may be viewed 
as twice the Lefschetz pencil of quartics.
Let $$\pi_{4,2}:X_{4,2} \subset \proj^3\times \proj^1 \rarr \proj^1$$
be the family obtained from a nonsingular Calabi-Yau hypersurface.
The family ${\pi}_{4,2}$ may also be viewed as twice the Lefschetz pencil.

\begin{Lemma} $n_{g,d}^{\widetilde{X}} = n_{g,d}^{X_{4,2}}.$
\end{Lemma}

\bpf
It suffices to prove the analogous statement for Gromov-Witten invariants.
Consider the degeneration of $X_{4,2}$ to the union 
$$X_{4,1} \cup_{K3} X_{4,1}$$
of two $(4,1)$ hypersurfaces along a smooth $K3$ surface.  
The degeneration formula of \cite{liruan,junli} implies 
$$N_{g,d}^{X_{4,2}} = 2 N_{g,d}^{X_{4,1}/K3}$$
where the latter term denotes the Gromov-Witten theory 
of $X_{4,1}$ relative to the $K3$ fiber.
Since the Gromov-Witten theory of $K3\times\proj^{1}$ 
vanishes, the trivial degeneration 
$$X_{4,1}\cup_{K3}(K3\times \proj^{1})$$
yields the equality of relative and absolute invariants
$$N_{g,d}^{X_{4,1}} = N_{g,d}^{X_{4,1}/K3}.$$

To study the small resolution $\widetilde{\pi}$, consider the family of double covers
$$\epsilon_{t}:C_{t} \mapsto \proj^{1}$$
ramified at $108$ generic points which specializes to our particular
double cover \eqref{f456} 
as $t \rightarrow 0$.  The behavior of Gromov-Witten theory in the conifold transition from 
$$X_{t} = \epsilon_{t}^{\ast}(X_{4,1})$$ to $\tilde{X}$ has been calculated by Li and Ruan \cite{liruan}:
$$N_{g,d}^{\widetilde{X}} = N_{g,d}^{X_{t}}.$$
By degenerating the base $C_{t}$ to two copies of $\proj^{1}$, we have a degeneration
of $X_{t}$ to two copies of $X_{4,1}$ attached at $54$ smooth $K3$ fibers.
As before, we apply the degeneration formula and the identification of relative and absolute invariants to obtain the equality
$$N_{g,d}^{\widetilde{X}}=N_{g,d}^{X_{t}} = 2N_{g,d}^{X_{4,1}} = N_{g,d}^{X_{4,2}}.$$
\epf

Instead of studying the Gromov-Witten invariants of $\widetilde{X}$, we
may study the Gromov-Witten invariants of $X_{4,2}$.

\subsection{Mirror symmetry}
\subsubsection{Overview}
The genus 0 invariants of $X_{4,2}$ are determined from hypergeometric series by
the mirror transformation. The mirror formulas of Candelas, 
de la Ossa, Green, and
Parkes
\cite{cogp} have been proven mathematically
in many settings \cite{giv1,giv2, lly}. In particular, the case of $X_{4,2}$ is understood
rigorously. We follow
the notation of \cite{pgiv}.

\subsubsection{Potential}
Let the variables $T_1,T_2$ correspond to the hyperplane classes 
$$H_1\subset \proj^3, \ \ H_2\subset \proj^1$$
respectively.
The genus 0 potential of $X_{4,2}$ for classes restricted from $\proj^3\times \proj^1$ is
$$\mathcal{F}(T_1,T_2)= 
\frac{1}{3}T_1^3 +2         T_1^2 T_2+
 \sum_{d_1,d_2\geq 0, \  (d_1,d_2)\neq (0,0)} N^{X_{4,2}}_{0,(d_1,d_2)}\ e^{d_1T_1} e^{d_2T_2}$$
where
we follow the  Gromov-Witten notation of Section \ref{gwrev}. The curve class $(d_1,d_2)$
is not a fiber class for $\pi^{4.2}$ if $d_2>0$.

\subsubsection{Hypergeometric series}
Let $t_1,t_2$ be new variables. 
Define the hypergeometric series $I_{i,j}(t_1,t_2)$ by
\begin{multline*}\label{ggt}
\sum_{i=0}^3 \sum_{j=0}^1 I_{i,j}(t_1,t_2) H_1^iH_2^j = \\
\sum_{d_1,d_2\geq 0} 
e^{(H_1+d_1)t_1} e^{(H_2+d_2)t_2}
\frac{\Pi_{r=0}^{4d_1+2d_2} (4H_1+2H_2+r)} {\Pi_{r=1}^{d_1}(H_1+r)^4  
\ \Pi_{r=1}^{d_2}(H_2+r)^2 
}. 
\end{multline*}
The right side, taken mod $H_1^4$ and $H_2^2$, is  valued in $H^*(\proj^3\times \proj^1,\mathbb{Q})$.
Formally,
$$I_{i,j}(t_1,t_2) \in \mathbb{Q}[[t_1,e^{t_1}, t_2, e^{t_2}]].$$
The functions $I_{i,j}(t)$ form a solution of the
Picard-Fuchs differential equation
associated to the mirror geometry. 

\subsubsection{Mirror transformation}
The mirror transformation is defined using two auxiliary functions.
Let 
$$F(e^{t_1},e^{t_2}) = \sum_{d_1=0}^\infty \sum_{d_2=0}^\infty
e^{d_1t_1} e^{d_2t_2} \frac{(4d_1+2d_2)!}{(d_1!)^{4}(d_2!)^2},$$
and let
$$G_{a,b}(e^{t_1},e^{t_2})= \sum_{d_1=0}^{\infty} \sum_{d_2=0}^\infty e^{d_1t_1} e^{d_2t_2} \frac{(4d_1+2d_2)!}
{(d_1!)^{4}(d_2!)^2}
\Big( \sum_{r=1}^{ad_1+bd_2} \frac{1}{r}\Big)$$
for $a,b\geq 0$. 

The mirror transformation relating the variables $T_i$ and $t_i$ is determined by the
following equations:
$$T_1= t_1+ \frac{4(G_{4,2}(e^{t_1}, e^{t_2})-G_{1,0}(e^{t_1},e^{t_2}))}{F(e^{t_1},e^{t_2})},$$
$$T_2= t_2+ \frac{2(G_{4,2}(e^{t_1}, e^{t_2})-G_{0,1}(e^{t_1},e^{t_2}))}{F(e^{t_1},e^{t_2})}.$$
Exponentiation yields
$$
e^{T_1}= e^{t_1}\cdot\text{exp}\left(
\frac{4(G_{4,2}(e^{t_1}, 
e^{t_2})-G_{1,0}(e^{t_1},e^{t_2}))}{F(e^{t_1},e^{t_2})}\right),
$$
$$
e^{T_2}= e^{t_2}\cdot\text{exp}\left(
\frac{2(G_{4,2}(e^{t_1}, e^{t_2})
-G_{0,1}(e^{t_1},e^{t_2}))}{F(e^{t_1},e^{t_2})}\right).
$$
Together, the above four equations
define a change of variables from formal series in $T_1, e^{T_1}, T_2, e^{T_2}$
to formal series in $t_1, e^{t_1}, t_2, e^{t_2}$. The mirror transformation
is easily seen to be invertible.

\subsubsection{Genus 0 invariants}
The genus 0 potential $\mathcal{F}$ is determined by mirror symmetry,
\begin{multline*}
\mathcal{F}(T_1(t_1,t_2), T_2(t_1,t_2)) =  \\
\left(\frac{2I_{1,1}-I_{2,0}}{I_{1,0}}\right)  \left( \frac{I_{3,0}}{I_{1,0}}\right)
  + 2\left(\frac{I_{2,0}}{I_{1,0}}\right)\left( \frac{I_{2,1}}{I_{1,0}}\right)  -2
\left(\frac{I_{3,1}}{I_{1,0}}\right).
\end{multline*}
The arguments of the functions on the right side are understood to be $t_1$ and $t_2$.
The genus 0 BPS states $n_{0,d}^{X_{4,2}}$ are determined by
$\mathcal{F}$.

\subsection{Proof of Theorem 2}
Consider twice the Lefschetz pencil of quartics 
$$\widetilde{\pi}: \widetilde{X} \rarr C_{53}.$$
Corollary 1 in genus 0 is
\begin{equation}\label{fd23}
n_{0,d}^{\widetilde{X}}= \sum_{h=0}^\infty 
r_{0,h}\cdot  NL_{h,d}^{\widetilde{\pi}}\ .
\end{equation}

We now solve for the Noether-Lefschetz numbers of $\tilde{\pi}$.
By \eqref{zz3z},
$$NL^{\widetilde{\pi}}_{h,d} = \phi^{\widetilde{\pi}}\left[\frac{\bigtriangleup_{4}(h,d)}{8}\right]$$
where
$\phi^{\widetilde{\pi}}(q)$ is a homogeneous polynomial of degree 21 in $A$ and $B$.
We need only 22 equations to determine $\phi^{\widetilde{\pi}}(q)$.
Using the mirror symmetry calculation of $n_{0,d}^{\widetilde{X}}$,
equation \eqref{fd23} provides infinitely many relations.
In particular, $\phi^{\widetilde{\pi}}(q)$ is easily
determined by linear algebra.

The precise formula for $\phi^{\widetilde{\pi}}$ is $2\Theta$ where
$\Theta$ is given in Section \ref{mform} since $\widetilde{\pi}$ is
twice the Lefschetz pencil of quartics.
The modular form $\Theta$ was first computed
in \cite{germans}.

\subsection{Modular identity}\label{modid}
Equation \eqref{fd23} may be viewed as a rather intricate
relation between hypergeometric functions (after mirror
transformation) on the left and modular forms on the right.
Let $$\mathcal{G}(q) = -\frac{2}{q} +168
+ \sum_{d \geq 1} n_{0,d}^{X_{4,2}} q^{\frac{d^{2}}{8}}$$
be the generating function determined by the property
$$\sum_{d=1}^\infty \sum_{k =1}^\infty  n_{0,d}^{X_{4,2}} \frac{1}{k^{3}}e^{dkT_1}
= \left( \mathcal{F}(T_1,T_2) - \frac{1}{3}T_{1}^{3} - 2T_{1}^{2}T_2 \right)|_{e^{T_2}=0}$$
where $\mathcal{F}$ is determined as above.
\begin{Corollary}
We have the equality
$$\mathcal{G}(q) = 2\frac{\Theta(q)}{\Delta(q)}\ ,$$
where $\Theta(q)$ is given in Section \ref{mform} and
$$\Delta(q) = q \prod_{n=1}^\infty (1-q^n)^{24}\ .$$
\end{Corollary}

Such relations are produced by Theorem 1 for many classical examples.  
For any $1$-parameter family of $K3$ surfaces obtained via a toric complete intersection, there 
is an associated identity of special functions.
The relation obtained from the STU model studied in
\cite{KMPS} is the Harvey-Moore identity. In fact, the
Harvey-Moore identity is the {\em only} one for which a direct
proof (avoiding Theorem 1) is known. The proof is
due to Zagier and can be found in \cite{KMPS}.

\subsection{Proof of Corollary 2} \label{c22}
Let $\pi$ be the Lefschetz pencil of quartic $K3$ surfaces.
The difference between $NL^\pi_{h,d}$ and
the degree of 
$$\overline{\mathcal{D}}_{h,d}\subset \proj(\text{Sym}^4(V^*))$$ is simply the
contribution of the nodal quartics.
The nodal quartics contribute to $NL^\pi_{h,d}$ but not 
the hypersurface $\overline{\mathcal{D}}_{h,d}$.

Using the relation $NL^\pi_{h,d}= \frac{1}{2}NL^{\widetilde{\pi}}_{h,d}$,
we can study instead the doubled family.
The Picard lattice of each of the 108 fibers of $\widetilde{\pi}$
corresponding to the
original nodal fibers of $\pi$ is
\begin{equation} \label{n2k}
\left( \begin{array}{cc}
4 & 0  \\
0 & -2  \end{array} \right).
\end{equation}
We use here the genericity of the Lefschetz pencil $\pi$.

The equation 
$\lan \beta, L \ran=d$ is solvable in the lattice
\eqref{n2k} if and only if $d$ is divisible by $4$.
Then, 
$\lan \beta, \beta \ran = 2h-2$ is solvable
if and only if 
$$4(\frac{d}{4})^2-2 n^2 =2h-2$$
in which case there are two solutions.
In the solvable cases,
$$\bigtriangleup_4(h,d)= 8n^2.$$
Hence, the contribution of the nodal fiber to the
Noether-Lefschetz numbers of $\widetilde{\pi}$ is
$$\Psi(q)= 108\cdot 2 \sum_{n>0} q^{n^2}.$$
The Corollary follows by halving. \qed

\section{Direct Noether-Lefschetz calculations}
\label{dnl}

\subsection{Overview}
We apply Corollary $3$ to directly study 
$K3$ surfaces of low degree via a more sophisticated approach
to modular forms.
The key idea is to
construct a basis of the space of vector-valued modular forms
of Corollary 3 instead of working with the much larger
space of scalar-valued modular forms as in Section \ref{qqk3}.
For many classical families, the dimensions of the
associated spaces of vector-valued modular forms are very small. 
The Noether-Lefschetz numbers can often be specified by
a few classical calculations.  In particular, we see another derivation of Theorem 2.

\subsection{Rankin-Cohen brackets}
Since each component of a vector-valued modular form is a half-weight modular 
form of level $2l$, we can use a basis of the latter
 to construct all vector-valued modular forms.  
In practice, however, the method is tedius since the
dimensions of the spaces of scalar-valued modular forms are
much larger.
We will instead apply the following shortcut for low degree $K3$ surfaces.

Let
$f(q)$ and $g(q)$ be scalar-valued level $N$ modular forms on the
upper-half plane $\mathcal{H}$ 
 of
 weights $k_1$ and $k_2$ respectively. 
For each integer $n\geq 0$, the $n$-th Rankin-Cohen bracket is a bilinear differential operator defined by the expression
$$[f(q), g(q)]_{n} = \sum_{r=0}^{n} (-1)^{r} \binom{n+k_1-1}{n-r} \binom{n+k_2-1}{r} f^{(r)}(q)\cdot g^{(n-r)}(q),$$
where $f^{(r)}$ denote $r$ applications of the differential operator 
$$\frac{d}{d\tau} = q\frac{d}{dq}\ .$$
For $n=0$, the $0$-th bracket is just multiplication.

The key feature of Rankin-Cohen brackets is the preservation of modularity.
Suppose we are given a representation $\rho$ of $Mp_{2}(\ZZ)$ on $V$, a modular form
$f \in \mathrm{Mod}(Mp_2(\ZZ),k_1,\rho)$ of weight $k_1$ and type $\rho$, and a scalar-valued modular form $g \in \mathrm{Mod}(SL_2(\ZZ), k_2)$ of weight $k_2$ and level $1$.  Let
$$f(q) = \sum_{\gamma} f_{\gamma}(q) v_{\gamma} \in V$$
denote the decomposition of $f$ into components with respect to some basis of $V$.
For each integer $n \geq 0$, the Rankin-Cohen bracket 
is a holomorphic function on $\mathcal{H}$ with values in $V$ defined by
$$[f,g]_{n}(q) = \sum_{\gamma} [f_{\gamma}(q), g(q)]_{n} v_{\gamma}.$$
We then have the following result.

\begin{Lemma}\label{rcbrackets}
$[f,g]_{n}(q) \in \mathrm{Mod}(Mp_2(\ZZ), k_1+k_2+2n,\rho)$.
\end{Lemma}
\begin{proof}
For scalar-valued modular forms, a proof is given in \cite{zagier}. 
Since $g$ is scalar-valued and level $1$, 
the same argument translates to the vector-valued context without change.
\end{proof}

\subsection{Bases of modular forms}

Following the notation of Corollary 3, 
we now look for modular forms of weight 21/2 and type $\rho_{l}^{*}$ 
for even 
$$l = 2, 4, 6, 8\ .$$
From the dimension formula given in Section \ref{picr} below,   
$$\dim(\mathrm{Mod}(Mp_{2}(\ZZ), 21/2,\rho_{l}^{*})) = 2, 3, 4, 5$$
for $l = 2, 4, 6, 8$ respectively.  We are only interested{\footnote{The
cusp condition is obtained from Borcherds' results and
was omitted in the statement of Corollary 3 for simplicity.}} in the subspace 
$$\mathrm{Mod}_0(Mp_2(\ZZ),21/2, \rho_{l}^{*})$$ 
of forms
$\sum f_{i}(q) v_{i}$ where $f_r(q)$ is a cusp form for $r \ne 0$.  In the 
$l=8$ case, we have a $4$-dimensional subspace.

We can use Rankin-Cohen brackets to construct explicit bases.
Indeed, for each $l$, there is a 
canonical weight $1/2$ modular form 
given by the Siegel theta function (see \cite{borch2}, Section $4$),
$$\theta^{(l)}(q) = \sum_{i=0}^{l-1} \sum_{s\in \mathbb{Z}}q^{\frac{(ls+i)^{2}}{2l}} v_{i} \in\mathrm{Mod}(Mp_{2}(\ZZ), 1/2, \rho_{l}^{*}).$$
Therefore, for $n= 0,1,2,3$, Lemma \ref{rcbrackets} gives us a modular form, 
$$F^{l}_{n}(q) = [\theta^{(l)}(q), E_{10-2n}(q)]_{n} 
\in \mathrm{Mod}(Mp_{2}(\ZZ), 21/2,\rho_{l}^{*}),$$
of weight $21/2$
where $E_{2k}(q)$ denotes Eisenstein series of weight $2k$.

Using the explicit formula for Rankin-Cohen brackets
and the dimension formula, the following Lemma 
is obtained by calculating the initial Taylor coefficients.
\begin{Lemma}
For $l=2,4,6$, the modular forms
$$F^{l}_{n}(q) = [\theta^{(l)}(q), E_{10-2n}(q)]_{n}, n = 0, \dots, l/2$$
form a basis of $ \mathrm{Mod}(Mp_{2}(\ZZ), 21/2,\rho_{l}^{*})$.  
For $l=8$, the modular
 forms for $n= 0, \dots, 3$ form a basis of the subspace $\mathrm{Mod}_{0}(Mp_2(\ZZ),21/2, \rho_{l}^{*})$.
\end{Lemma}

\subsection{Classical families of $K3$ surfaces}

A general $K3$ surface of degree $l=2,4,6, 8$ is either a 
branched cover of $\proj^2$ (for $l=2$) or a complete intersection 
in projective space.  We obtain 1-parameter families of 
quasi-polarized $K3$ surfaces of degree $l$ by taking a 
generic Lefschetz pencil of these constructions 
(and resolving singularities as discussed in Section \ref{quar}).
Because the space of vector-valued
forms is of low dimension, 
we only need a few classical constraints to 
completely determine the associated modular form.
In fact, we will use only the following constraints:
\begin{enumerate}
\item[(i)] 
the degree of the Hodge bundle $R^{2}\pi_{*}\mathcal{O}$
(the coefficient of $q^{0}v_{0}$),
\item[(ii)] the number of nodal fibers (the coefficient of $q^{1}v_{0}$),
\item[(iii)]
vanishing  obtained from  Castelnuovo's bound 
in Lemma \ref{castelnuovo} below.
\end{enumerate}

The following result is a special case of Castelnuovo's bound for
projective curves
\cite{acgh}.

\begin{Lemma}\label{castelnuovo}
Given a $K3$ surface with very ample bundle $L$ and an primitive
 curve class $\beta$,
we have the inequality 
$$\lan\beta,\beta\ran \leq 2\binom{L\cdot\beta -1}{2} - 2\ .$$
\end{Lemma}

We now apply these constraints for 1-parameter families of $K3$
given by Lefschetz pencils for $l=2,4,6,8$.

\vspace{6 mm}

\noindent $\bullet$  \textit{Degree $2$ $K3$ surfaces}
\vspace{6 mm}

A generic $K3$ surface of degree 2
 is a double cover of $\proj^{2}$ branched along a nonsingular
 sextic plane curve.  
Consider
a family 
$$R\subset \proj^1 \times \proj^2$$ 
of sextics defined by a generic hypersurface of
type $(2,6)$.
Let $X$ be 
the double cover of $\proj^1 \times \proj^2$
ramified over $R$.
Since all the singular fibers of
$$R\rarr\proj^1$$ 
are irreducible and 
nodal, 
the associated family  
$$\pi: X \rarr \proj^1$$
of $K3$ surfaces is
smooth except for finitely many fibers with nodal singularities.

The degree of the Hodge bundle is $-1$ by a Riemann-Roch calculation. 
 The number of nodal fibers of $\pi$
is $150$, twice
the degree of the discriminant locus of sextics.  
Since we have a 2-dimensional space of forms, 
the
generating series of 
 Noether-Lefschetz numbers is the vector-valued modular form
$$\overrightarrow{\Theta}(q) = -F^{(2)}_{0}(q) - \frac{1}{2}F^{(2)}_{1}(q).$$

In the case of $l=2$, the discriminant $\Delta$ of a rank $2$ lattice with degree $2$ polarization determines the coset class $\delta$ by $\delta = \Delta \mod 2$.  So there is no loss of information if we replace $\overrightarrow{\Theta}(q)$ by the sum of the components 
$\Theta(q) = \overrightarrow{\Theta}_0 + \overrightarrow{\Theta}_1$.

If we consider the theta functions
$$U = \sum_{n \in \ZZ} q^{n^{2}/4},\ \ \
 V = \sum_{n\in\ZZ}(-1)^{n}q^{n^{2}/4},$$
we can express $\Theta$ as a polynomial function of $U$ and $V$:
\begin{align*}
\Theta(q) &= \frac{1}{1024}(U^{21} - 12 U^{17} V^4 -402 U^{13} V^{8} - 572 U^{9}V^{12} - 39U^{5}V^{16})\\
&=  -1 + 150 q + 1248 q^{5/4} + 108600 q^2 + 332800 q^{9/4} + 5113200 q^3\cdots.
\end{align*}
To see equivalence of the two expressions, we  observe 
 both are modular forms of weight $21/2$ with respect to $\Gamma(4)$ and check 
the agreement of sufficiently many coefficients. 
\vspace{6 mm}

\noindent
$\bullet$  \textit{Degree $4$ $K3$ surfaces}

\vspace{6 mm}

A generic $K3$ surface 
of degree 4 
is a quartic hypersurface in $\proj^3$.  
If we take a generic Lefschetz pencil of such quartics, 
the degree of the Hodge bundle is $-1$.  
Using Lemma \ref{castelnuovo}, the Noether-Lefschetz degrees
associated to the lattices
$$ \left(
  \begin{array}{ c c }
     4 & 1 \\
     1 & 0
  \end{array} \right), \ \ 
  \left(
  \begin{array}{ c c }
     4 & 2 \\
     2 & 0
  \end{array} \right)
$$
both vanish.  
Indeed, by choosing a generic pencil, we can assume
all fibers containing these Picard lattices have
very ample quasi-polarization.   
The coefficients of $q^0v_0, q^{1/8}v_{1},$ and $q^{1/2}v_{2}$ determine 
$$\overrightarrow{\Theta}(q) = - F^{(4)}_{0}(q) - \frac{5}{4}F^{(4)}_{1}(q) - \frac{16}{21}F^{(4)}_{2}(q).$$

Again, as in
the degree $2$ case, we can recover all Noether-Lefschetz degrees from 
$$\Theta(q) =  \overrightarrow{\Theta}_0(q) + \overrightarrow{\Theta}_1(q)+ \overrightarrow{\Theta}_2(q).$$
In terms of 
$$A = \sum_{n \in \ZZ} q^{n^{2}/8},\ \ \ 
 B = \sum_{n\in\ZZ}(-1)^{n}q^{n^{2}/8},$$
we recover the expression for $\Theta(q)$ given in Section \ref{mform}
since both are modular forms of weight  $21/2$ and level $8$ which 
agree on initial terms.
\vspace{6 mm}

\noindent $\bullet$  \textit{Degree $6$ $K3$ surfaces}

\vspace{6 mm}

A generic $K3$ surface of degree 6
 is the intersection of a quadric and cubic hypersurface in $\proj^4$.  
We have two basic families.  
We can fix a quadric and take a Lefschetz pencil of cubics or vice versa.  
In each case, we have vanishings associated to the lattices
$$ \left(
  \begin{array}{ c c }
     6 & 1 \\
     1 & 0
  \end{array} \right), \ \ 
  \left(
  \begin{array}{ c c }
     6 & 2 \\
     2 & 0
  \end{array} \right)
$$
from the Castelnuovo bound.
Along with the Hodge bundle degree and the number of nodal fibers, 
we completely determine the Noether-Lefschetz series.  

For the first family, the Hodge and nodal degrees are $-1$ and $98$ 
respectively. We obtain the series
$$\overrightarrow{\Theta}(q) = -F^{(6)}_{0}(q) - \frac{49}{24}F^{(6)}_{1}(q) - \frac{8}{3}F^{(6)}_{2}(q) - \frac{12}{5}F^{(6)}_{3}(q).$$
For the second family, the
Hodge and nodal degrees are $-1$ and $7$. We obtain the series
$$\overrightarrow{\Theta}(q) = -F^{(6)}_{0}(q) - \frac{17}{8}F^{(6)}_{1}(q) - \frac{22}{7}F^{(6)}_{2}(q) - \frac{18}{5}F^{(6)}_{3}(q).$$

One can read off other classical calculations from our results.  
For example, 
the number of surfaces containing elliptic plane curves or 
containing lines are the Noether-Lefschetz degrees associated to the lattices
$$
 \left(
  \begin{array}{ c c }
     6 & 3 \\
     3 & 0
  \end{array} \right),\ \ 
   \left(
  \begin{array}{ c c }
     6 & 1 \\
     1 & -2
  \end{array} \right)
$$
respectively.  In the first family, the degrees are $0$ and $168$
respectively. In the second family, the degrees are $10$ and $198$.  
In both cases, the numbers agree with earlier enumerative calculations.
\vspace{6 mm}

\noindent $\bullet$  \textit{Degree $8$ $K3$ surfaces}

\vspace{6mm}
A generic $K3$ surface of degree 8
is the intersection of three quadric hypersurfaces in $\proj^{5}$.  
The basic family comes from fixing two quadrics 
and allowing the third to vary in a Lefschetz pencil.  
Again, the series is determined by the Hodge term, 
the nodal term, and the two Castelnuovo vanishings from 
Lemma \ref{castelnuovo}.  The Hodge term is given by $-1$,
and the number of nodal fibers is $80$. We find
$$\overrightarrow{\Theta}(q) = -F^{(8)}_{0}(q) - \frac{49}{18}F^{(8)}_{1}(q) - \frac{128}{27}F^{(8)}_{2}(q) - \frac{256}{45}F^{(8)}_{3}(q).$$
Again, we can read off that the number of fibers containing 
a line is $128$, agreeing with the classical calculation.

\vspace{6 mm}
For all the classical examples discussed above, the mirror
symmetry calculation of the genus 0 Gromov-Witten 
invariants is solvable in terms of hypergeometric
functions. In each case, Theorem 1 yields a remarkable
identity with hypergeometric functions (after mirror
transformation) on the left and 
modular forms on the right, as in Section \ref{modid}.

The lower Noether-Lefschetz degrees in the above
classical examples can often be pursued by alternative methods.
In particular, matches with our modular form calculations 
have been found in \cite{R,V}.

\section{Picard rank of ${\mathcal{M}}_l$} \label{picr}
The Picard ranks of the moduli spaces
of quasi-polarized $K3$ surfaces
$\mathcal{M}_{l}$ are unknown. By an argument of O'Grady,
the ranks can grow arbitrarily large \cite{ogrady}.
Let 
\begin{equation}\label{tt23}
\mathrm{Pic}(\mathcal{M}_l)^{NL} \otimes \Q \subset
\mathrm{Pic}(\mathcal{M}_{l})\otimes \Q
\end{equation}
denote the span of the Noether-Lefschetz divisors $D_{h,d}$.
We make the following conjecture.

\begin{Conjecture}\label{hhhppp}
The inclusion is an isomorphism,
$$\mathrm{Pic}(\mathcal{M}_l)^{NL} \otimes \Q \cong
\mathrm{Pic}(\mathcal{M}_{l})\otimes \Q.$$
\end{Conjecture}

Bruinier has calculated
the dimension of the space $\mathrm{Pic}(\mathcal{M}_l)^{NL} \otimes \Q$
in equations (6-7) of \cite{bruinier}. 
If Conjecture \ref{hhhppp}  holds, we obtain a formula for the Picard rank of $\mathcal{M}_l$.

We now recount Bruinier's formula for the span of the Noether-Lefschetz
divisors.
By Borcherds' work, we have a map
\begin{equation} \label{dgj5}
\mathrm{Mod}(Mp_{2}(\Z), 21/2, \rho_{l}^{\ast})^{\ast} 
\rightarrow \mathrm{Pic}(\mathcal{M}_{l})\otimes \com.
\end{equation}
Let $\mathrm{Cusp}(Mp_{2}(\Z), 21/2, \rho_{l}^{\ast})$ denote the subspace
of cusp forms --- modular forms
for which the Fourier coefficients $c_{0,\gamma}$ vanish for all $\gamma$.
The map  \eqref{dgj5} induces a map
\begin{equation}
\label{ggyy55}
\mathrm{Cusp}(Mp_{2}(\Z),21/2, \rho_{l}^{\ast})^{\ast} \rightarrow 
(\mathrm{Pic}(\mathcal{M}_{l})\otimes\com)/ \com  K,
\end{equation}
where $K$ is the Hodge bundle on $\mathcal{M}_{l}$.
Bruinier shows the map \eqref{ggyy55}
 is injective \cite{bruinier}.  Specifically,
if $L$ is a $(2,n)$ lattice containing
two copies of $U$ as direct summands,
Bruinier shows that every relation among Heegner divisors is obtained from
Borcherds' theta lifting. Therefore,
$$\mathrm{dim} \ \mathrm{Pic}(\mathcal{M}_l)^{NL} \otimes \Q = 
1+ \mathrm{dim}\ \mathrm{Cusp}(Mp_{2}(\Z), \rho_{l}^{\ast}, 21/2).$$

A direct calculation of the dimension of the
space of cusp forms via Riemann-Roch  yields the following
evaluation
 \cite{bruinier}: 
\begin{eqnarray*}
\text{dim }\mathrm{Pic}(\mathcal{M}_{l})^{NL}\otimes \Q &=& \ \ 
1+\frac{31}{24} + \frac{31}{48}l 
-\frac{1}{8\sqrt{l}}\mathrm{Re}(G(2, 2l))\\
& & - \frac{1}{6\sqrt{3l}}\mathrm{Re}(e^{-2\pi i \frac{19}{24}}(G(1,2l) + G(-3, 2l))) \\ & & 
- \sum_{k=0}^{l/2} \left\{\frac{k^2}{2l}\right\} - C,
\end{eqnarray*}
where $G(a,b)$ denotes the quadratic Gauss sum
$$G(a,b) = \sum_{k=0}^{b-1} e^{-2\pi i\frac{ak^2}{b}},$$
the braces $\{,\}$ denote fractional part, 
and $C$ is the cardinality of the set
$$\left\{ k\  \left| \right. \  0 \leq k \leq \frac{l}{2}, \frac{k^2}{2l}\in \Z\right\}.$$
For $l=2,4,6$, the formula yields 
$$\text{dim }\mathrm{Pic}(\mathcal{M}_{l})^{NL}\otimes \Q = 2,3,4$$ 
respectively.
For $l=2$ and $4$, we have agreement
with the Picard ranks of $\mathcal{M}_l$
calculated in \cite{kirw,shah1, shah2}.
Hence, the inclusion \eqref{tt23} is an isomorphism in at least the
first two
cases.

\vspace{+14 pt}
\noindent
Department of Mathematics \\
Princeton University \\

\vspace{+8pt}
\noindent{\em Current addresses:}
\vspace{+10pt}

\noindent
Department of Mathematics \\
Columbia University \\
dmaulik@math.columbia.edu \\

\noindent
Departement Mathematik\\
ETH Z\"urich\\
rahul@math.ethz.ch

\end{document}